\documentclass[a4paper,11pt,reqno]{amsart}

\usepackage[
    left=1in,
    right=1in,
    top=1in,
    bottom=1in,
    marginparwidth=0.75in,
    marginparsep=0.1in
]{geometry}

\usepackage{mathpazo}  

\usepackage{amsmath, amsthm, amssymb, mathtools}
\usepackage{mathrsfs}
\usepackage{bm}
\usepackage{empheq}

\numberwithin{equation}{section}
\allowdisplaybreaks

\usepackage{xcolor}

\definecolor{mycite}{RGB}{197,25,210}   
\definecolor{mylink}{RGB}{8,119,127}    
\definecolor{myurl}{RGB}{225,119,1}     

\usepackage[
    colorlinks,
    citecolor=mycite,
    linkcolor=mylink,
    urlcolor=myurl,
 allcolors=black,
    backref=page
]{hyperref}

\usepackage{aliascnt}

\usepackage[
final
]{showkeys}

\DeclarePairedDelimiterX{\tnorm}[1]
{\lVert\mkern-2mu|}{\rVert\mkern-2mu|}{#1}

\newcommand{\iu}{\mathrm{i}}
\newcommand{\cas}{\mathfrak{c}\mathfrak{a}\mathfrak{s}}

\usepackage{graphicx}
\usepackage{subcaption}

\usepackage{tikz-cd}
\usetikzlibrary{positioning, calc}

\usepackage{float}
\usepackage{booktabs}

\usepackage{enumitem}
\usepackage{lineno}
\usepackage{calc}

\usepackage{listings}

\theoremstyle{definition}

\newtheorem{theorem}{Theorem}[section]

\newaliascnt{definition}{theorem}
\newtheorem{definition}[definition]{Definition}
\aliascntresetthe{definition}

\newaliascnt{proposition}{theorem}
\newtheorem{proposition}[proposition]{Proposition}
\aliascntresetthe{proposition}

\newaliascnt{lemma}{theorem}
\newtheorem{lemma}[lemma]{Lemma}
\aliascntresetthe{lemma}

\newaliascnt{corollary}{theorem}
\newtheorem{corollary}[corollary]{Corollary}
\aliascntresetthe{corollary}

\newaliascnt{example}{theorem}

\aliascntresetthe{example}

\newaliascnt{remark}{theorem}
\newtheorem{remark}[remark]{Remark}
\aliascntresetthe{remark}

\begin{document}

\title[LDP for cubic NLS: {\em subcritical case}]
{
Large Deviations for the Nonlinear Schr\"odinger Equation with Randomized Quasi-Periodic Initial Data in Higher Dimensions
: {\em Subcritical Case} 
}

\author{Fei Xu}
\address{\scriptsize (F. Xu)~Institute of Mathematics, Jilin University, Changchun 130012, P.R. China.}
\email{\color{magenta}stuxuf@outlook.com}
\thanks{The first author (F.~Xu) was supported in part by 
the National Natural Science Foundation of China (Grant No.~12501235), 
the China Postdoctoral Science Foundation (Grant Nos.~BX20240138 and 2025M783104), 
and the Jilin Provincial Postdoctoral Science Foundation (2025). 
He would like to thank David Damanik (Rice University) for sharing several papers on {\em Wave Turbulence Theory}, Gigliola Staffilani (MIT) for discussions on {\em higher-order normal form}, fundamental insights into the nature of the {\em criticality} (onset of the first nonlinear effect), and valuable comments on earlier versions of the paper, Yvain Bruned (Universit\'e de Lorraine), Yuzhao Wang (University of Birmingham) and Xinyu Zhao (New Jersey Institute of Technology) for related discussions.
}

\author{Yong Li}
\address{\scriptsize  (Y. Li)~Institute of Mathematics, Jilin University, Changchun 130012, P.R. China. 
School of Mathematics and Statistics, Center for Mathematics and Interdisciplinary Sciences, Northeast Normal University, Changchun, Jilin 130024, P.R. China.}
\email{\color{magenta}liyong@jlu.edu.cn}
\thanks{The second author (Y.~Li) is the corresponding author. He was supported in part by the National Basic Research Program of China (Grant No. 2013CB834100), and the National Natural Science Foundation of China (Grant Nos. 12471183 and 12531009).}

\date{\today}



\maketitle

\begin{abstract}

We study the cubic weakly nonlinear Schr\"odinger equation with randomized spatially quasi-periodic initial data in higher dimensions. Under a polynomial decay assumption in Fourier space, we establish a \emph{Large Deviations Principle} for rogue waves in the so-called \emph{subcritical} time regime. 

The proof proceeds in two main steps. We first characterize the distribution of the linear solution and establish the corresponding linear large deviations principle. The lower bound is obtained via pointwise estimates, while the upper bound follows from a combination of truncation and probabilistic arguments.  {The method used in this step appears to be {\em new}; compare with \cite{GGKS23}.} We then perform a detailed combinatorial analysis of the Picard iteration, deriving an effective bound for the Duhamel term and thereby establishing the nonlinear large deviations principle.

\end{abstract}

\tableofcontents

\section{Introduction}

\subsection{Random Data Theory and Rogue Waves as Extreme Events}

The probabilistic study of dispersive partial differential equations originates from the seminal works of Lebowitz--Rose--Speer and Bourgain on invariant measures for the nonlinear Schr\"odinger equation; see \cite{LRS88,Bourgain94,Bourgain96}. These developments initiated a systematic investigation of {\em low-regularity dynamics through randomization} of either forcing terms or initial data.

From a structural point of view, such problems naturally split into two classes. The first concerns stochastic partial differential equations (SPDEs), where randomness enters through singular forcing and leads to dynamics such as the KPZ equation \cite{H13AoM} and stochastic quantization models in Euclidean quantum field theory \cite{GIP15}. The second class consists of deterministic dispersive equations with random initial data, which provide a probabilistic framework for studying low-regularity behavior and are closely connected to wave turbulence theory \cite{Zakharov,BGHS21}. In many singular settings, a common feature is that the underlying analysis requires {\em renormalization} to handle divergent nonlinear interactions. 

In the past decade, there has been remarkable progress in the analysis of singular stochastic {\em parabolic} equations. The development of the theory of {\em regularity structures} by Hairer \cite{H14,BHZ19} and {\em paracontrolled calculus} by Gubinelli–Imkeller–Perkowski \cite{GIP15} has led to the local well-posedness theory, essentially completing the picture in the so-called subcritical regime; see \cite{DNY22Inventiones} for further introduction. 

More recently, the emerging theories of {\em random tensors} \cite{DNY22Inventiones} and {\em random averaging operators} \cite{DNY24AoM}, developed by Deng–Nahmod–Yue, can be interpreted as {\em dispersive} counterparts to the aforementioned parabolic frameworks. These theories provide a systematic frameworks to the renormalization of oscillatory multilinear interactions; see also the corresponding {\em hyperbolic} counterpart \cite{BDNY24Inventiones}.

Random data theory has played a fundamental role in the study of the long-time behaviour of solutions to evolution partial differential equations. One of the most active directions in this field is the derivation of the {\em Wave Kinetic Equation} (WKE), which connects the statistical behaviour of Fourier coefficients of solutions to an effective kinetic description. Significant recent advances in this direction can be found in \cite{DH22,DH23Inventiones,DH23,DHM24,DIP25}, together with the references therein.

Another interesting direction related to random data theory concerns the so-called {\em rogue waves}, a term used by oceanographers to describe {\em isolated} and {\em large-amplitude} waves, typically defined as those whose height exceeds twice the {significant wave height} of the ambient sea state. Empirically, such events occur more frequently than predicted by Gaussian statistics; see \cite{PR13,DGVE,GGKS23}.

To explain this phenomenon, several mechanisms for rogue wave formation have been proposed, although no unified theory is currently available. {\em Linear superposition} and {\em nonlinear focusing} are two fundamental mechanisms widely studied in the literature. The former is a constructive interference phenomenon: the phases of many weakly interacting waves align at a certain point in space, producing a large-amplitude peak. The latter suggests that rogue waves may emerge from a significant exchange of energy between waves of different wavenumbers, leading to a substantial amplification of one component; see \cite{G25}. In this setting, rogue waves are often viewed as limiting cases of {\em breather solutions} and may also arise from {\em soliton collisions}; see \cite{PR13}. These represent {\em non-generic} behaviours and may be regarded as {\em rare phenomena}, or more generally as {\em extreme events} from a probabilistic perspective; see \cite{GGKS23}. 
 
{The probabilistic perspective suggests that} rogue waves provide a canonical, yet still poorly understood, example of extreme events. A central problem is therefore to {\em quantify deviations of the wave-height distribution from Gaussianity}, and thereby to estimate the likelihood of rogue wave occurrence. This remains a {\em long-standing problem} with significant implications for ships and naval structures; see \cite{DGVE}.

In \cite{DGVE}, Dematteis, Grafke and Vanden-Eijnden observed that {\em rogue waves obey a large deviations principle}. This phenomenon was later rigorously established by Garrido et al. for the cubic weakly nonlinear Schr\"odinger (NLS) equation 
\begin{align}\label{eq:snls}
\mathrm{i}\partial_t u + \partial_{xx} u + {\varepsilon^{\alpha}} |u|^{2} u = 0,\quad \alpha>1, 
\end{align}
with randomized periodic initial data of the form
\[
u(0,x) = \sum_{n \in \mathbb{Z}} c(n)\, g_{n}\, e^{\mathrm{i} n x}.
\]
{Their result holds over time scales $t \sim \varepsilon^{-\beta}$ for both subcritical and critical 
cases, that is, $\alpha - 1 > \beta$ and $\alpha - 1 = \beta > 0$ respectively. 
}
Their analysis relies on an exponential decay assumption on the deterministic Fourier components of the random initial data, namely $c(n)=a e^{-b|n|}$ or $c(n)=a e^{-b|n|^2}$ for all $n \in \mathbb{Z}$, where $a,b>0$ are fixed. This assumption plays a crucial role in their argument; see \cite{GGKS23} for further details. Recently, this condition has been relaxed. In \cite{LW25}, Liang and Wang extended the result to $\ell^1(\mathbb{Z})$ decay, while Fan and Ye \cite{FL25} further improved it to $c_n = \langle n \rangle^{-\frac{1}{2}+}$.

In contrast to standard NLS \eqref{eq:snls}, Grande studied the following beating NLS
\begin{align}
{\iu}\partial_t u + \partial_{xx} u = 2\cos(2x)\,|u|^2 u,
\end{align}
with randomized periodic initial data supported on two Fourier modes of the form
\[
u(0,x) = \varepsilon \bigl(\alpha e^{\iu x} + \beta e^{-\iu x}\bigr),
\]
where $\alpha$ and $\beta$ are complex-valued independent Gaussian random variables with zero mean and variances
$\sigma_\alpha^2 = \mathbb{E}\bigl[|\alpha|^2\bigr]$ and $\sigma_\beta^2 = \mathbb{E}\bigl[|\beta|^2\bigr]$. When $\sigma^2_\alpha \neq \sigma^2_\beta$, he showed that resonant energy exchange between Fourier modes leads to a {\em fattening} of the tails of the probability distribution of the sup-norm of the solution, thereby increasing the likelihood of rogue wave formation. Consequently, a large deviations principle holds; see \cite{G25}. 

The above discussion has been confined to the spatially periodic setting in the study of rogue waves. As noted by Wilkening and Zhao \cite{WZ21JNS}, rogue waves are also intimately connected with {\em spatially quasi-periodic dynamics} arising from {\em modulational instability} (also known as {\em Benjamin--Feir instability}). Thus, {\em extending the study of rogue waves to the spatially quasi-periodic setting under randomization} remains a natural and largely unexplored direction. In addition, the deterministic Cauchy problem with spatially quasi-periodic initial data forms an important and active area in nonlinear partial differential equations; see, for example, the works of Deift \cite{Deift07,Deift17} and subsequent works citing them.

\subsection{Randomized Quasi-Periodic Fourier Series on $\mathbb{R}^d$ and Initial Data}\label{rqp}
Before formulating the problem precisely, we first introduce in this subsection the randomization of spatially quasi-periodic Fourier series on $\mathbb R^d$.

A function $f:\mathbb{R}^d \to \mathbb{C}$ is said to be {\em quasi-periodic} if it is quasi-periodic in each spatial direction $x_j$, with associated rationally independent frequency vector $\omega_j \in \mathbb{R}^{\nu_j}$ for $j=1,\dots,d$, that is, 
\begin{align}\label{eq:rationallyindependent}
\langle n_j,\omega_j \rangle= 0 \quad \Longrightarrow \quad n_j = 0 \in \mathbb{Z}^{\nu_j}.
\end{align}

Let $\nu := \sum_{j=1}^d \nu_j$ denote the total number of frequencies. We define the {\em frequency matrix} $\boldsymbol{\Omega} \in \mathbb{R}^{\nu \times d}$ (with $\nu > d$) by 
\begin{align}\label{eq:frequencymatrix}
\boldsymbol{\Omega}
= \mathrm{diag}\bigl(\omega_1^{\top}, \dots, \omega_d^{\top}\bigr)
= \begin{pmatrix}
\omega_1^{\top} & & \\
& \ddots & \\
& & \omega_d^{\top}
\end{pmatrix}.
\end{align}
We say that $\boldsymbol{\Omega}$ is {{\em non-resonant} or} \emph{rationally independent} if \eqref{eq:rationallyindependent} holds for each direction. Equivalently, in a more compact form,
\[
\boldsymbol{n}\boldsymbol{\Omega} = \boldsymbol{0}
\quad \Longrightarrow \quad 
\boldsymbol{n} = \boldsymbol{0} \in \mathbb{Z}^{\nu},
\]
where $\boldsymbol{n} = (n_{1},\dots,n_d) \in \mathbb{Z}^{\nu_1} \times \cdots \times \mathbb{Z}^{\nu_d} \simeq \mathbb{Z}^{\nu}$.

The block-diagonal structure of $\boldsymbol{\Omega}$ reflects the independence of the frequency components across different spatial directions. In particular, initial data \eqref{eq:initialdata} is quasi-periodic in each spatial direction $x_j$ with frequency vector $\omega_j$. Therefore, $f$ admits the {\em quasi-periodic Fourier expansion}
\begin{equation}\label{eq:fourierseries}
    f(\boldsymbol{x}) = \sum_{\boldsymbol{n} \in \mathbb{Z}^\nu} \hat{f}(\boldsymbol{n}) \, e^{\mathrm{i}\langle \boldsymbol{n}\boldsymbol{\Omega}, \boldsymbol{x} \rangle}, \quad \boldsymbol{x} \in \mathbb{R}^d,
\end{equation}
where the Fourier coefficients $\hat{f}(\boldsymbol{n})$ are determined by the spatial average
\begin{equation}\label{eq:fouriercoefficient}
    \hat{f}(\boldsymbol{n}) 
    = \lim_{L \to \infty} \frac{1}{(2L)^d} 
    \int_{[-L,L]^d} f(\boldsymbol{x}) \, e^{-\mathrm{i}\langle \boldsymbol{n}\boldsymbol{\Omega}, \boldsymbol{x} \rangle} \, \mathrm{d}\boldsymbol{x}.
\end{equation}

By introducing the angular variables $\boldsymbol{y} = \boldsymbol{x}\boldsymbol{\Omega}^\top \pmod{2\pi}$ on the torus $\mathbb{T}^\nu$, we define the {\em generating function} $F$ as the following Fourier expansion on
\begin{equation*}
    F(\boldsymbol{y}) = \sum_{\boldsymbol{n} \in \mathbb{Z}^\nu} \hat{f}(\boldsymbol{n}) e^{\mathrm{i} \langle \boldsymbol{n}, \boldsymbol{y}\rangle}, \quad \boldsymbol{y} \in \mathbb{T}^\nu,
\end{equation*}
such that $f(\boldsymbol{x}) = F(\boldsymbol{x}\boldsymbol{\Omega}^\top)$.

By the {\em Birkhoff Ergodic Theorem}, the rational independence of frequency matrix $\boldsymbol{\Omega}$ ensures that spatial average in \eqref{eq:fouriercoefficient} coincides with the {\em Haar integral} over the torus, 
that is,
\begin{equation*}
    \hat{f}(\boldsymbol{n}) = \frac{1}{(2\pi)^\nu} \int_{\mathbb{T}^\nu} F(\boldsymbol{y}) e^{-\mathrm{i} \langle \boldsymbol{n}, \boldsymbol{y} \rangle} \,\mathrm{d}\boldsymbol{y}.
\end{equation*}

In this paper, we consider the {\em randomization} of Fourier series \eqref{eq:fourierseries} by prescribing the coefficients as
\begin{equation*}
    \hat{f}(\boldsymbol{n}) = c(\boldsymbol{n}) g_{\boldsymbol{n}}, \quad \boldsymbol{n} \in \mathbb{Z}^\nu.
\end{equation*}
\begin{itemize}
\item Here $\{c(\boldsymbol{n})\}_{\boldsymbol{n}\in\mathbb{Z}^\nu}$ is a deterministic sequence satisfying the following polynomial decay condition: 
\begin{align}\label{eq:decay}
    |c(\boldsymbol{n})| \le \tnorm{\boldsymbol{n}}_{-\rho}
    = \prod_{j=1}^{d}\prod_{j'=1}^{\nu_j} (1+|n_{j,j'}|)^{-\rho_{j,j'}},
    \end{align}
    where
    \[
    \boldsymbol{n} = (n_{1,1},\cdots,n_{1,\nu_1};\cdots;n_{d,1},\cdots,n_{d,\nu_d}) \in \mathbb{Z}^{\nu_1} \times \cdots \times \mathbb{Z}^{\nu_d},
    \]
    and
    \[
    \rho = (\rho_{1,1},\cdots,\rho_{1,\nu_1};\cdots;\rho_{d,1},\cdots,\rho_{d,\nu_d}) \quad \text{with } \rho_{j,j'} > 0 ~~\text{large enough for all indices}.
    \]
\item The random component $\{g_{\boldsymbol{n}}\}_{\boldsymbol{n} \in \mathbb{Z}^\nu}$ constitutes a family of independent and identically distributed (i.i.d. for short) standard complex Gaussian random variables satisfying the following conditions:
\begin{equation*}
    \mathbb{E}[g_{\boldsymbol{n}}] = 0, \qquad \mathbb{E}[g_{\boldsymbol{n}}g_{\boldsymbol{m}}] = 0, \qquad\mathbb{E}[g_{\boldsymbol{n}} \overline{g_{\boldsymbol{m}}}] = \delta_{\boldsymbol{n},\boldsymbol{m}},
\end{equation*}
where $\delta_{\boldsymbol{n},\boldsymbol{m}}$ denotes the Kronecker delta. Equivalently,  $\operatorname{Re}\, g_{\boldsymbol{n}}$ and $\operatorname{Im}\, g_{\boldsymbol{n}}$ are independent $\mathcal{N}_{\mathbb{R}}(0,\tfrac{1}{2})$ random variables; see \autoref{sec:complexgaussian} for further details. 
\end{itemize}

This construction defines a Gaussian random field $F$ on the torus $\mathbb{T}^\nu$, which induces a randomized quasi-periodic field $f$ on $\mathbb{R}^d$, that is, 
\begin{equation}\label{eq:randomqp_d}
    f(\boldsymbol{x}) = \sum_{\boldsymbol{n} \in \mathbb{Z}^\nu} c(\boldsymbol{n}) g_{\boldsymbol{n}} e^{\mathrm{i} \langle \boldsymbol{n}\boldsymbol{\Omega}, \boldsymbol{x} \rangle}.
\end{equation}
\begin{proposition}[Expectation and Variance]\label{thm:expectation-variance}
Let randomized field $f(\boldsymbol{x})$ be defined as in \eqref{eq:randomqp_d}. Then the following properties hold:
\begin{enumerate}
\item[\rm(i)] The field is centered: $\mathbb{E}[f(\boldsymbol{x})] = 0$ for all $\boldsymbol{x} \in \mathbb{R}^d$.
\item[\rm(ii)] If the decay rates satisfy ${\rho_{j,j'} > 1/2}$ for each $j=1,\dots,d$ and $j'=1,\dots,\nu_j$, then the second moment is uniformly bounded: $\mathbb{V}[f(\boldsymbol{x})] < \infty$.
\end{enumerate}
\end{proposition}

\begin{proof}
\textbf{(i)} By linearity of the expectation operator,
\begin{align*}
    \mathbb{E}[f(\boldsymbol{x})] = \sum_{\boldsymbol{n} \in \mathbb{Z}^\nu} c(\boldsymbol{n}) e^{\mathrm{i} \langle \boldsymbol{n}\boldsymbol{\Omega}, \boldsymbol{x} \rangle} \mathbb{E}[g_{\boldsymbol{n}}].
\end{align*}
Since $\{g_{\boldsymbol{n}}\}$ are centered, that is, $\mathbb{E}[g_{\boldsymbol{n}}] = 0$ for all $\boldsymbol{n} \in \mathbb{Z}^\nu$, it follows that $\mathbb{E}[f(\boldsymbol{x})] = 0$.

\vspace{0.5em}
\textbf{(ii)} For the second moment, we use $|z|^2 = z \overline{z}$. By independence and orthogonality, we obtain
\begin{align*}
    \mathbb{E}\left[|f(\boldsymbol{x})|^2\right]  &= \mathbb{E} \left[ \sum_{\boldsymbol{n},\boldsymbol{m} \in \mathbb{Z}^\nu} c(\boldsymbol{n}) \overline{c(\boldsymbol{m})} g_{\boldsymbol{n}} \overline{g_{\boldsymbol{m}}} e^{\mathrm{i} \langle (\boldsymbol{n}-\boldsymbol{m})\boldsymbol{\Omega}, \boldsymbol{x} \rangle} \right] \nonumber \\
    &= \sum_{\boldsymbol{n},\boldsymbol{m} \in \mathbb{Z}^\nu} c(\boldsymbol{n}) \overline{c(\boldsymbol{m})} e^{\mathrm{i} \langle (\boldsymbol{n}-\boldsymbol{m})\boldsymbol{\Omega}, \boldsymbol{x} \rangle} \delta_{\boldsymbol{n},\boldsymbol{m}} \nonumber \\
    &= \sum_{\boldsymbol{n} \in \mathbb{Z}^\nu} |c(\boldsymbol{n})|^2. \label{eq:parseval_iso}
\end{align*}
Substituting polynomial decay assumption \eqref{eq:decay} yields
\begin{align*}
\mathbb{E}\left[|f(\boldsymbol{x})|^2\right] 
&= \sum_{\boldsymbol{n} \in \mathbb{Z}^\nu} |c(\boldsymbol{n})|^2 \\
&\le \sum_{\substack{n_{j,j'} \in \mathbb{Z} \\ j=1,\dots,d \\ j' = 1,\dots,\nu_j}} 
   \prod_{j=1}^d \prod_{j'=1}^{\nu_j} \left(1+|n_{j,j'}|\right)^{-2\rho_{j,j'}} \\
&= \prod_{j=1}^d \prod_{j'=1}^{\nu_j} \sum_{n_{j,j'} \in \mathbb{Z}} \left(1+|n_{j,j'}|\right)^{-2\rho_{j,j'}}.
\end{align*}
Each one-dimensional summation converges provided $\rho_{j,j'} > 1/2$. Consequently, the product is finite, ensuring that the second moment (and thus the variance) is finite and uniform with respect to $\boldsymbol{x}$.
\end{proof}

\begin{proposition}[Tail bound for $g_{\boldsymbol{n}}$]\label{thm:poly}
Let $0 < \delta \ll 1$, and let
\[
\kappa=(\kappa_{1,1},\dots,\kappa_{1,\nu_1};\dots;\kappa_{d,1},\dots,\kappa_{d,\nu_d}),
\qquad \kappa_{j,j'}>0~~\text{for all}~~j=1,\cdots,d~~\text{and}~~ 1\le j^\prime\leq\nu_j.
\]
Define the local sets
\[
\Omega_{\delta,\boldsymbol{n}} =
\Bigl\{ \omega \in \Omega :
|g_{\boldsymbol{n}}(\omega)|
>
{\delta^{-\frac{1}{2}}}\tnorm{\boldsymbol{n}}_\kappa
\Bigr\},
\]
where
\[
\tnorm{\boldsymbol{n}}_\kappa
= \prod_{j=1}^d \prod_{j'=1}^{\nu_j} (1 + |n_{j,j'}|)^{\kappa_{j,j'}},
\]
and the global set
\[
\Omega_\delta = \bigcup_{\boldsymbol{n} \in \mathbb{Z}^\nu} \Omega_{\delta,\boldsymbol{n}}.
\]
Then we have
\[
\mathbb{P}(\Omega_\delta) =\mathcal O( e^{-\delta^{-1}}).
\]
\end{proposition}
\begin{proof}
\textbf{Step 1. Tail probability of $g_{\boldsymbol{n}}$.}

Since $g_{\boldsymbol{n}} \sim \mathscr{N}_{\mathbb{C}}(0,1)$, we have
$|g_{\boldsymbol{n}}|^2 \sim \mathscr E(1)$. 
Therefore,
\[
\mathbb{P}(\Omega_{\delta,\boldsymbol{n}})
=
\exp\bigl(-\delta^{-1}\tnorm{\boldsymbol{n}}_{\kappa}^2\bigr).
\]
Since $\tnorm{\boldsymbol{n}}_{\kappa} \ge 1$ and
$\tnorm{\boldsymbol{n}}_{\kappa} = 1$ if and only if $\boldsymbol{n}=0$, we have
\[
\mathbb{P}(\Omega_{\delta,\boldsymbol{0}})=e^{-\delta^{-1}}
\quad\text{and}\quad
\tnorm{\boldsymbol{n}}_{\kappa}^2 - 1 > 0 \ \text{for } \boldsymbol{n}\neq 0.
\]
Thus,
\[
\mathbb{P}(\Omega_\delta)
\le e^{-\delta^{-1}} + \sum_{\boldsymbol{n}\neq 0}
\exp\bigl(-\delta^{-1}\tnorm{\boldsymbol{n}}_{\kappa}^2\bigr),
\]
and equivalently,
\[
\mathbb{P}(\Omega_\delta)
=
e^{-\delta^{-1}}
\Biggl(
1 +
\sum_{\boldsymbol{n}\neq 0}
\exp\bigl(-\delta^{-1}(\tnorm{\boldsymbol{n}}_{\kappa}^2-1)\bigr)
\Biggr).
\]

\textbf{Step 2. Positive gap.}

Since $\tnorm{\boldsymbol{n}}_\kappa \to \infty$ as $|\boldsymbol{n}|\to\infty$
and $\tnorm{\boldsymbol{n}}_\kappa > 1$ for all $\boldsymbol{n}\neq 0$,
the quantity $\tnorm{\boldsymbol{n}}_\kappa^2 - 1$ attains a positive minimum on 
$\mathbb{Z}^\nu \setminus \{0\}$. Hence there exists $c_0>0$ such that
\[
\tnorm{\boldsymbol{n}}_\kappa^2 - 1 \ge c_0.
\]
Define the level sets
\[
L_t
=
\Bigl\{
\boldsymbol{n}\in\mathbb{Z}^\nu:
c_0+t \le \tnorm{\boldsymbol{n}}_{\kappa}^2 - 1 < c_0+t+1
\Bigr\},
\quad t\in\mathbb{N}.
\]
Then
\[
\sum_{\boldsymbol{n}\neq 0}
e^{-\delta^{-1}(\tnorm{\boldsymbol{n}}_{\kappa}^2-1)}
=
\sum_{t\ge 0}\sum_{\boldsymbol{n}\in L_t}
e^{-\delta^{-1}(\tnorm{\boldsymbol{n}}_{\kappa}^2-1)}.
\]

\textbf{Step 3. Counting argument.}

For $\boldsymbol{n}\in L_t$, we have
\[
\prod_{j,j'} (1+|n_{j,j'}|)^{2\kappa_{j,j'}}
< c_0+t+2.
\]
Taking logarithms and using positivity of all $\kappa_{j,j'}$, it follows that
each coordinate satisfies
\[
1+|n_{j,j'}|
< (c_0+t+2)^{\frac{1}{2\kappa_{j,j'}}}.
\]
Hence, for any fixed \(1 \leq j \leq d\) and \(1 \leq j' \leq \nu_j\), the corresponding \(n_{j,j'}\) satisfies
\begin{align*}
n_{j,j'} 
\leq 2\left( (c_0 + t + 2)^{\frac{1}{2\kappa_{j,j'}}} - 1 \right) + 1 
= 2(c_0 + t + 2)^{\frac{1}{2\kappa_{j,j'}}} - 1 
< 2(c_0 + t + 2)^{\frac{1}{2\kappa_{j,j'}}}.
\end{align*}
Thus
\[
|L_t|
<2^\nu
(c_0+t+2)^{\sum_{j=1}^d \sum_{j'=1}^{\nu_j} \frac{1}{2\kappa_{j,j'}}}.
\]

\textbf{Step 4. Summation and conclusion.}

We estimate
\[
\sum_{t\ge 0} |L_t| e^{-\delta^{-1}(c_0+t)}
\lesssim
e^{-c_0\delta^{-1}}
\sum_{t\ge 0} (c_0+t+2)^{\sum_{j=1}^d \sum_{j'=1}^{\nu_j} \frac{1}{2\kappa_{j,j'}}} e^{-t\delta^{-1}}.
\]
Since $e^{-t\delta^{-1}}$ decays exponentially fast,
\[
\sum_{t\ge 0} (c_0+t+2)^{\sum_{j=1}^d \sum_{j'=1}^{\nu_j} \frac{1}{2\kappa_{j,j'}}} e^{-t\delta^{-1}} \lesssim \delta^{1+\sum_{j=1}^d \sum_{j'=1}^{\nu_j} \frac{1}{2\kappa_{j,j'}}}.
\]
Therefore,
\[
\mathbb{P}(\Omega_\delta)
\lesssim e^{-\delta^{-1}}
\left(1 + \delta^{1+\sum_{j=1}^d \sum_{j'=1}^{\nu_j} \frac{1}{2\kappa_{j,j'}}}e^{-c_0\delta^{-1}}\right)
\lesssim e^{-\delta^{-1}}.
\]
\end{proof}

{
\begin{corollary}\label{cor:polynomial_decay}
Assume that the following conditions hold:
\begin{itemize}
\item \textbf{(Related to initial data in Fourier space)} 
Let $\rho$ and $\kappa$ be as in \eqref{eq:decay} and \autoref{thm:poly}, respectively. Assume that
\[
\rho_{j,j'} - \kappa_{j,j'} > 2, \quad j=1,\dots,d,\; j'=1,\dots,\nu_j.
\]

\item \textbf{(Size of initial data with respect to $\varepsilon$)} 
Let $\omega \notin \Omega_{\delta}$ with $\delta\sim\varepsilon^{1+\eta}$, where
\begin{equation}\label{eq:delta}
0 \le \eta < \min\left\{ 2\Bigl(\frac{1}{\nu}-\mu_1\Bigr)\eta_2-\mu_2,\; \eta_3\right\}. 
\end{equation}
Here $0<\mu_1 \ll 1/\nu, 0<\mu_2 \ll 1$, $\eta_2$ and $\eta_3$ are given by
\begin{align}
\eta_2 &= \sum_{j=1}^{d} \sum_{j'=1}^{\nu_j}
\bigl(\rho_{j,j'} - \kappa_{j,j'} - 1\bigr), \label{eta2}\\
\eta_3 &= \frac{2}{3}(\alpha-1)-\frac{\mu_3}{3}, 
\quad 0<\mu_3 \ll 1.
\end{align}
\end{itemize}
 Then the randomized Fourier coefficients satisfy the following almost sure polynomial decay estimates:
\begin{align}\label{eq:decayall}
\bigl|c(\boldsymbol{n}) g_{\boldsymbol{n}}(\omega)\bigr|
\le \varepsilon^{-\frac{1}{2}-\frac{\eta}{2}} \,\tnorm{\boldsymbol{n}}_{-(\rho-\kappa)},\quad\omega\notin\Omega_\delta,
\end{align}
where
\begin{align}\label{eq:rhokappa}
\rho - \kappa 
= \bigl( \rho_{1,1}-\kappa_{1,1}, \dots, \rho_{1,\nu_1}-\kappa_{1,\nu_1}; \dots; \rho_{d,1}-\kappa_{d,1}, \dots, \rho_{d,\nu_d}-\kappa_{d,\nu_d} \bigr).
\end{align}
\end{corollary}

\begin{remark}[$\rho-\kappa$]
Roughly speaking, the condition on $\rho$ and $\kappa$ requires $\rho_{j,j'} - \kappa_{j,j'}$ to be sufficiently large to ensure absolute convergence of series appearing in estimates such as \autoref{size1}.
\end{remark}
\begin{remark}[$\eta$]
Regarding the assumption on $\eta$, there are three remarks below. 

\begin{enumerate}
\item [(a)] The condition $\eta\ge0$ ensures that the second term in \eqref{ee} is asymptotically negligible compared to the first term. If this condition fails, for example, $\delta = \varepsilon^{1/2}$, then
\[
\varepsilon \log\left(1 + \frac{\mathbb P(\Omega_\delta)}{\mathbb P(\mathscr A_\varepsilon^{(3)})}\right)
\sim
\varepsilon \log\left(1 + e^{-\frac{1}{\sqrt{\varepsilon}} + \frac{1}{\varepsilon}}\right)
\sim1-\sqrt{\varepsilon}\to 1,\quad\varepsilon\to0^+.
\]
%
Hence, the contribution of the bad set is no longer negligible at the scale of large deviations. Similar analysis to \eqref{w2} and \eqref{w3}. 

\item[(b)] The condition $\eta < 2\left(\frac{1}{\nu}-\mu_1\right)\eta_2-\mu_2$ is required in order to get the effective scale $\mathcal O(\varepsilon^{-\frac{1}{2}+\frac{\mu_2}{2}})$ of $\mathscr R_N$; see \autoref{lemma:vanishing}. 

\item[(c)] The condition $\eta < \eta_3$ is imposed to obtain the effective scale
$\mathcal{O}(\varepsilon^{-\frac{1}{2}+\frac{\mu_3}{2}})$ for the Duhamel estimates by working on a shorter time scale, namely after reducing the original time scale by a factor of
$\varepsilon^{\frac{\eta}{2}+\frac{\mu_3}{2}}$; see \autoref{eq:duhamel}.
    
\end{enumerate}
\end{remark}
}

\subsection{Problem Statement and Main Results}
Motivated by the preceding discussion, we consider the Cauchy problem on $\mathbb{R}^d$ for a weakly nonlinear Schr\"odinger equation (NLS) with randomized quasi-periodic initial data:
\begin{empheq}[left=\empheqlbrace]{align}
    &\mathrm{i}\partial_t u + \Delta u + {\varepsilon^{\alpha}} |u|^{2} u = 0,\quad\alpha>1; \label{eq:nls}\\
    &u(0, \boldsymbol{x}) = \sum_{\boldsymbol{n} \in \mathbb{Z}^\nu} 
        c(\boldsymbol{n})\, g_{\boldsymbol{n}}\, 
        e^{\mathrm{i} \langle \boldsymbol{n}\pmb{\Omega}, \boldsymbol{x} \rangle}. \label{eq:initialdata}
\end{empheq}

Here $u = u(t,\boldsymbol{x})$ is a complex-valued function with $(t,\boldsymbol{x}) \in \mathbb{R} \times \mathbb{R}^d$. 
The operator $\partial_t$ denotes the time derivative, and $\Delta = \sum_{j=1}^d \partial_j^2$ is the Laplacian, where $\partial_j = \frac{\partial}{\partial x_j}$. 
The parameter $0 < \varepsilon \ll 1$ measures the strength of the nonlinearity, and $\mathrm{i}$ denotes the imaginary unit satisfying $\mathrm{i}^2 = -1$. 


For model \eqref{eq:nls}-\eqref{eq:initialdata}, we study the corresponding rogue waves introduced above. Although evolution equation \eqref{eq:nls} is deterministic, randomness in initial data \eqref{eq:initialdata} induces a stochastic process \(t \mapsto u(t,\boldsymbol{x})\) for each fixed \(\boldsymbol{x} \in \mathbb{R}^d\), and a random field \(\boldsymbol{x} \mapsto u(t,\boldsymbol{x})\) for each fixed \(t \in \mathbb{R}\).

We quantify such extreme events by introducing the tail event
\begin{equation}\label{eq:rogue}
    \mathscr{A}_\varepsilon^\infty = \left\{ \sup_{\boldsymbol{x} \in \mathbb{R}^d} |u(t, \boldsymbol{x})| > z_0 {\varepsilon^{-\frac{1}{2}}} \right\},
\end{equation}
where $z_0 > 0$ is a fixed threshold. This event describes the occurrence of a wave of amplitude at least $z_0 \varepsilon^{-1/2}$ at a fixed time $t>0$.

The main objective of this paper is to describe the asymptotic behavior of the probability
$
\mathbb{P}(\mathscr{A}_\varepsilon^\infty)
$
as $\varepsilon \to 0^+$.

Our main result is a nonlinear large deviations principle in the subcritical time regime (\autoref{thm:nonlinearldp}), together with a linear large deviations principle (\autoref{thm:linearldp}) and a well-posedness result (\autoref{thm:eude}).

\begin{theorem}[Large Deviations Principle]\label{thm:nonlinearldp}
Consider randomized quasi-periodic Cauchy problem \eqref{eq:nls}--\eqref{eq:initialdata} together with \autoref{cor:polynomial_decay}, and the rogue wave event $\mathscr{A}_\varepsilon^\infty$ defined in \eqref{eq:rogue}. 
If the observation time satisfies
\[
t =\mathcal O\left(\varepsilon^{-\beta}\right), \quad\beta=\alpha-1-\frac{3}{2}\eta-\frac{\mu_3}{2}.
\]
Then we have
\[
\lim_{\varepsilon\to0^+}
\varepsilon \log \mathbb{P}\big(\mathscr{A}_\varepsilon^\infty\big)
=
-\frac{z_0^2}{\sum_{\boldsymbol n\in\mathbb Z^\nu} |c(\boldsymbol n)|^2}.
\]
\end{theorem}

\begin{remark}[Function Space]
To the best of our knowledge, this is the first large deviations result on randomization for the spatially quasi-periodic data Cauchy problem in the framework of dispersive equations.

In contrast to decaying or periodic functions, quasi-periodic functions exhibit persistent oscillations at infinity, which makes the analysis considerably more delicate than in the classical settings.

\end{remark}

\begin{remark}[Time Scale]
{
Clearly, $\alpha - 1 = \beta+\frac{3}{2}\eta+\frac{\mu_3}{2}>\beta$, which corresponds to the so-called \emph{subcritical} regime in the sense of \cite{GGKS23}, namely the regime preceding the onset of the first nonlinear effect; see \cite{s26}. 

} 
\end{remark}

\begin{remark}[Decay Assumption]
In \cite{GGKS23}, Garrido et al. established a large deviations principle for the periodic cubic NLS in one spatial dimension under an exponential decay assumption in Fourier space, or rather, on the determining component $\{c_n\}_{n\in\mathbb Z}$ of Fourier coefficients. They further raised the question of identifying the optimal decay condition under which the LDP holds.

Recently, this assumption has been relaxed. In \cite{LW25}, Liang and Wang improved the result to $\ell^1(\mathbb Z)$ decay, while Fan and Ye \cite{FL25} further extended it to $c_{n\in\mathbb Z} = \langle n \rangle^{-(\frac{1}{2}+\theta)}$ for any $\theta>0$. 

In the present paper, we consider the randomized quasi-periodic Cauchy problem \eqref{eq:nls}--\eqref{eq:initialdata} under a polynomial decay assumption of the form \eqref{eq:decay} based on \cite{DLX24JFA}. It is unclear whether the corresponding exponent is optimal, and determining the sharp decay threshold ensuring the validity of the large deviations principle in the quasi-periodic setting remains open.
\end{remark}
\begin{remark}[Proof]
(a)~The proof of \autoref{thm:nonlinearldp} consists of three main steps.
\begin{itemize}
\item We first establish a linear LDP for the associated linear problem. In this step, we analyze the explicit law of the linear solution, in particular its modulus squared. The lower bound follows from pointwise estimates, while the upper bound is obtained via global control in polar coordinates, combined with analytic and probabilistic arguments; see \autoref{sec:linearldp} and \autoref{fig:linearproof}.

\item We then carry out a well-posedness analysis and derive Duhamel estimates on the relevant time scale. To this end, we employ a combinatorial argument (see \cite{C07,DG16,X25}) to identify the relevant time scale and derive decay properties of the solution; see \autoref{sec:wellp}.

\item Finally, we establish the nonlinear LDP by combining Duhamel's formula with the above estimates, thereby completing the proof; see \autoref{sec:nonlinear}.
\end{itemize}

(b)~In the one-dimensional periodic setting, that is, on the torus $\mathbb T$, we compare our method with that of \cite{GGKS23}.
\begin{itemize}
\item In the linear setting, the analysis of~\cite{GGKS23} is based on the G\"artner--Ellis theorem, reducing the derivation of the large deviation principle (LDP) to the computation of the {\em cumulant generating function}. The associated {\em rate function} is then obtained through the {\em Fenchel--Legendre transform}, yielding the linear LDP. Moreover, their argument relies on specific number-theoretic properties of $\mathbb Z$, whose analogues in the quasi-periodic setting seem unavailable.
    
    In contrast, we directly exploit the explicit distribution of the squared modulus of the linear solution, combined with truncation and probabilistic arguments, which is new and simplifies the original argument. Independently, Liang and Wang also provided a new proof, avoiding combinatorial arguments and instead relying on analytic methods, for the linear LDP within the framework of the G\"artner-Ellis theorem; see \cite{LW25}.
\end{itemize}
\end{remark}

\subsection{Outline of the paper}
In \autoref{sec:linearldp}, we establish a linear LDP for the associated linear problem. In \autoref{sec:wellp}, we prove a well-posedness result, including the relevant time scale and Fourier decay properties. In \autoref{sec:nonlinear}, we derive the Duhamel estimates and establish the nonlinear LDP. In \autoref{appendixps}, we collect some basic facts on probability theory related to complex Gaussian random variables.

\subsection{Notation}
We summarize some notation used throughout the paper below.
\begin{table}[htpb]
\centering
\renewcommand{\arraystretch}{1.2}
\begin{tabular}{ll}
\toprule
Notation & Meaning \\
\midrule
${}^\top$ 
& transpose (vectors are written in row form) \\

$\langle \boldsymbol{y},\boldsymbol{z}\rangle$ 
& $\boldsymbol{y}\boldsymbol{z}^{\top}=\sum_{j=1}^N y_j z_j$ \\

$A=\mathcal{O}(B)\ (\text{or } A\lesssim B)$ 
& $|A|\le CB$ for some $C>0$ independent of parameters \\

$A=\mathcal{O}_D(B)\ (\text{or } A\lesssim_D B)$ 
& $|A|\le C_D B$, where $C_D$ may depend on $D$ \\

%


$a+$ 
& $a+\mu$, where $0<\mu\ll1$ arbitrarily small \\

$a-$ 
& $a-\mu$, where $0<\mu\ll1$ arbitrarily small \\

\bottomrule
\end{tabular}
\end{table}

\section{Linear LDP}\label{sec:linearldp}

Consider the linear problem obtained by setting \(\varepsilon = 0\) in \eqref{eq:nls}, namely
\begin{equation*}
\begin{cases}
\mathrm{i}\partial_t u + \Delta u = 0,\\
u(0,\boldsymbol{x}) = \displaystyle \sum_{\boldsymbol{n} \in \mathbb{Z}^\nu} c(\boldsymbol{n})\, g_{\boldsymbol{n}}\, e^{\mathrm{i}\langle \boldsymbol{n}\boldsymbol{\Omega}, \boldsymbol{x} \rangle}.
\end{cases}
\end{equation*}
The solution is given explicitly by random Fourier series
\begin{equation}\label{eq:linearsolution}
u_{\mathrm{linear}}(t,\boldsymbol{x})
= \sum_{\boldsymbol{n} \in \mathbb{Z}^\nu} c(\boldsymbol{n})\, g_{\boldsymbol{n}}\,
e^{-\mathrm{i} t Q(\boldsymbol{n})}\,
e^{\mathrm{i}\langle \boldsymbol{n}\boldsymbol{\Omega}, \boldsymbol{x} \rangle},
\end{equation}
where the dispersion relation (neglecting the minus) is
\[
Q(\boldsymbol{n})
= \sum_{j=1}^d \big\langle n_j,\omega_j \big\rangle^2,
\qquad
\boldsymbol{n}=(n_1,\cdots,n_d)\in\mathbb Z^{\nu_1}\times\cdots\times\mathbb{Z}^{\nu_d}.
\]

\begin{theorem}[Linear LDP]\label{thm:linearldp}
Replacing \(u\) in \eqref{eq:rogue} by linear evolution \eqref{eq:linearsolution}, we obtain the following large deviations principle:
\begin{align}\label{eq:linearldp}
\lim_{\varepsilon \to 0^+} \varepsilon \log \mathbb{P}(\mathscr{A}_\varepsilon^\infty)
= -\frac{z_0^2}{\sum_{\boldsymbol{n} \in \mathbb{Z}^\nu} |c(\boldsymbol{n})|^2}.
\end{align}
\end{theorem}

\begin{remark}
For the proof of \autoref{thm:linearldp}, we obey \autoref{fig:linearproof}. 
\begin{figure}[htbp]
\centering
\begin{tikzpicture}[
    node distance=1.5cm and 3.0cm,  
    block/.style={
        rectangle, draw, thick, 
        inner sep=3pt,  
        align=center,
        font=\small,     
        minimum width=2.1cm,
        minimum height=0.7cm
    },
    arrow/.style={
        thick, -stealth,
        shorten >=1pt, shorten <=1pt
    },
    midbox/.style={
        rectangle, draw, thin, black, 
        inner sep=4pt,
        font=\scriptsize,  
        align=center,
        fill=none
    },
    newbox/.style={
        rectangle, draw, thin, black,
        inner sep=4pt,
        minimum width=2.2cm,
        font=\scriptsize,
        align=center,
        fill=none
    }
]

\node[block] (bottom) {Distribution of Linear Solution:\\[1mm] $u_{\text{linear}}(t,\boldsymbol x)\sim\mathscr N_{\mathbb C}\left(0,2\sigma^2\right)$};
\node[block, above=of bottom] (star) {Distribution of Modulus Squared:\\[1mm] $|u_{\text{linear}}(t,\boldsymbol x)|^2\sim\mathscr E\left(\frac{1}{2\sigma^2}\right)$};

\coordinate[above=1.5cm of star] (merge);

\node[block, above left=1.2cm and 3.0cm of merge] (lower) {Lower Bound};
\node[block, above right=1.2cm and 3.0cm of merge] (upper) {Upper Bound};

\coordinate[above=1.5cm of $(lower)!0.5!(upper)$] (v_line);

\node[block, above=of v_line] (linear) {Linear LDP};

\draw[arrow] (bottom.north) -- (star.south);
\draw[arrow] (star.north) -- (merge);
\draw[arrow] (merge) -- (lower.south) coordinate[midway] (mid_lower);
\draw[arrow] (merge) -- (upper.south) coordinate[midway] (mid_upper);
\draw[arrow] (lower.north) -- (v_line);
\draw[arrow] (upper.north) -- (v_line);
\draw[arrow] (v_line) -- (linear.south);

\path let \p1=(merge), \p2=(lower.south), \n1={atan2(\y2-\y1, \x2-\x1)+180} in
      node[midbox, rotate=\n1, anchor=north, yshift=-3pt, xshift=2pt] (pointwise_box) at (mid_lower) {Pointwise Control\\ $\mathscr A^{(0)}_\varepsilon\subset\mathscr A_\varepsilon^\infty$};

\path let \p1=(merge), \p2=(upper.south), \n1={atan2(\y2-\y1, \x2-\x1)} in
      node[midbox, rotate=\n1, anchor=north, yshift=-3pt, xshift=-2pt] (global_box) at (mid_upper) {Global Control\\ $\mathscr A_\varepsilon^\infty\subset\mathscr A_\varepsilon^{(1)}$};

\coordinate (global_box_e) at (global_box.south);
\coordinate (branch_point) at ($(global_box_e) + (1.5cm, -0.5cm)$);  
\draw[arrow]  (branch_point)--(global_box_e);

\node[newbox, above right=0.2cm and 0.75cm of branch_point] (new_box1) {High Modes: \\[1mm] Tail of Gaussian + \\
Decay of Deterministic Weight};
\node[newbox, below right=0.2cm and 0.75cm of branch_point] (new_box2) {Low Modes: \\[1mm] H\"older +\\
Pointwise Control + Chernoff Bound};

\draw[arrow] (new_box1.west)--(branch_point) ;
\draw[arrow] (new_box2.west)--(branch_point) ;

\end{tikzpicture}
\caption{Outline of the proof of \autoref{thm:linearldp}. We first derive the explicit distribution of the linear solution and its modulus squared; see \autoref{sec:dis}. Based on this representation, we obtain a lower bound for the rogue wave probability via pointwise estimates; see \autoref{sec:lowerbound}. We then establish the corresponding upper bound by working in polar coordinates. For the upper bound, we perform truncation of the solution into low and high frequency components. 
For the high modes, the tail of Gaussian and decay of the coefficients yield an effective control; see \autoref{sec:upperbound}.
For the low modes, we apply H\"older's inequality to separate the deterministic and random components: the deterministic part is controlled by decay of the Fourier coefficients, while the random component has an explicit distribution and is estimated via pointwise analysis (lower bound) and Chernoff bound (upper bound). }
\label{fig:linearproof}
\end{figure}

\end{remark}

\subsection{Distribution of the Linear Solution}\label{sec:dis}

In this subsection, we study the distribution of the linear solution and its modulus squared.

Since $\{g_{\boldsymbol{n}}\}_{\boldsymbol{n} \in \mathbb{Z}^\nu}$ are i.i.d.\ standard complex Gaussian random variables, the linear evolution preserves Gaussianity. In particular, $u_{\mathrm{linear}}(t,\boldsymbol{x})$ is a complex Gaussian random variable for every $(t,\boldsymbol{x}) \in \mathbb{R} \times \mathbb{R}^d$. Moreover, it is centered, i.e., 
\[
\mathbb{E}\big[u_{\mathrm{linear}}(t,\boldsymbol{x})\big]=0.
\]
In addition, it follows from $\mathbb{E}[g_{\boldsymbol{n}} g_{\boldsymbol{m}}] = 0$ and $\mathbb{E}[g_{\boldsymbol{n}} \overline{g_{\boldsymbol{m}}}] = \delta_{\boldsymbol{n},\boldsymbol{m}}$ that
\begin{align*}
\mathbb{E}\big[ (u_{\mathrm{linear}}(t, \boldsymbol{x}))^2 \big]
&= \mathbb{E} \Bigg[ \sum_{\boldsymbol{n}, \boldsymbol{m} \in \mathbb{Z}^\nu}
c(\boldsymbol{n}) c(\boldsymbol{m}) \,
g_{\boldsymbol{n}} g_{\boldsymbol{m}} \,
e^{-\mathrm{i} t (Q(\boldsymbol{n})+Q(\boldsymbol{m}))}
e^{\mathrm{i} \langle (\boldsymbol{n}+\boldsymbol{m})\boldsymbol{\Omega}, \boldsymbol{x} \rangle}
\Bigg] \nonumber \\
&= \sum_{\boldsymbol{n}, \boldsymbol{m} \in \mathbb{Z}^\nu}
c(\boldsymbol{n}) c(\boldsymbol{m}) \,
\mathbb{E}\big[g_{\boldsymbol{n}} g_{\boldsymbol{m}}\big] \,
e^{-\mathrm{i} t (Q(\boldsymbol{n})+Q(\boldsymbol{m}))}
e^{\mathrm{i} \langle (\boldsymbol{n}+\boldsymbol{m})\boldsymbol{\Omega}, \boldsymbol{x} \rangle} \nonumber \\
&=0,
\end{align*}
and
\begin{align*}
\mathbb{E}\big[ |u_{\mathrm{linear}}(t, \boldsymbol{x})|^2 \big]
&= \mathbb{E} \Bigg[ \sum_{\boldsymbol{n}, \boldsymbol{m} \in \mathbb{Z}^\nu}
c(\boldsymbol{n}) \overline{c(\boldsymbol{m})} \,
g_{\boldsymbol{n}} \overline{g_{\boldsymbol{m}}} \,
e^{-\mathrm{i} t (Q(\boldsymbol{n})-Q(\boldsymbol{m}))}
e^{\mathrm{i} \langle (\boldsymbol{n}-\boldsymbol{m})\boldsymbol{\Omega}, \boldsymbol{x} \rangle}
\Bigg] \nonumber \\
&= \sum_{\boldsymbol{n}, \boldsymbol{m} \in \mathbb{Z}^\nu}
c(\boldsymbol{n}) \overline{c(\boldsymbol{m})} \,
\mathbb{E}\big[g_{\boldsymbol{n}} \overline{g_{\boldsymbol{m}}}\big] \,
e^{-\mathrm{i} t (Q(\boldsymbol{n})-Q(\boldsymbol{m}))}
e^{\mathrm{i} \langle (\boldsymbol{n}-\boldsymbol{m})\boldsymbol{\Omega}, \boldsymbol{x} \rangle} \nonumber \\
&= \sum_{\boldsymbol{n} \in \mathbb{Z}^\nu} |c(\boldsymbol{n})|^2\nonumber\\
&\triangleq 2\sigma^2.
\end{align*}

Consequently, for each $(t,\boldsymbol{x}) \in \mathbb{R} \times \mathbb{R}^d$, the linear solution $u_{\mathrm{linear}}(t,\boldsymbol{x})$ is distributed as a centered complex Gaussian
\[
u_{\mathrm{linear}}(t,\boldsymbol{x}) \sim \mathcal{N}_{\mathbb{C}}(0,2\sigma^2),
\]
with variance independent of both time and space.

By the properties of complex Gaussian random variables, the normalized modulus squared 
follows a chi-squared distribution with two degrees of freedom, which is equivalent to an exponential distribution with rate parameter \(1/2\). That is, 
\[
\sigma^{-2} |u_{\mathrm{linear}}(t,\boldsymbol{x})|^2\sim\chi_2^2=\mathscr E\left(\frac{1}{2}\right).
\]
Equivalently, \(|u(t,\boldsymbol{x})|^2\) is exponentially distributed with rate parameter \(\frac{1}{2\sigma^2}\), i.e.
\[
|u_{\mathrm{linear}}(t,\boldsymbol{x})|^2 \sim \mathscr{E}\!\left(\frac{1}{2\sigma^2}\right),
\]
with probability density function
\[
f(x) = \frac{1}{2\sigma^2}\, e^{-\frac{x}{2\sigma^2}} \boldsymbol{1}_{[0,\infty)}(x).
\]

\subsection{Lower Bound}\label{sec:lowerbound}


Recall that \(u_{\mathrm{linear}}(t,\boldsymbol{x})\) is a centered complex Gaussian random field with variance \(2\sigma^2\), invariant in both time and space. Consequently, for every fixed \((t,\boldsymbol{x})\in\mathbb R\times\mathbb R^d\), the random variable \(|u_{\mathrm{linear}}(t,\boldsymbol{x})|^2\) follows an exponential distribution. This enables us to obtain a lower bound through the analysis of its pointwise tail probability, which we carry out in this subsection. 

To this end, for any fixed \((t,\boldsymbol{x}) \in \mathbb{R} \times \mathbb{R}^d\), we define the event
\begin{equation*}
    \mathscr{A}^{(0)}_\varepsilon
    =
    \left\{
    |u_{\mathrm{linear}}(t,\boldsymbol{x})|
    >
    z_0\,\varepsilon^{-1/2}
    \right\}.
\end{equation*}

Since the supremum over the spatial domain dominates any pointwise value, we have \(\mathscr{A}^{(0)}_\varepsilon \subset \mathscr{A}_\varepsilon^\infty\). Consequently,
\begin{align*}
    \mathbb{P}(\mathscr{A}_\varepsilon^\infty)
    &\geq \mathbb{P}(\mathscr{A}^{(0)}_\varepsilon) \\
    &= \mathbb{P}\big(|u_{\mathrm{linear}}(t,\boldsymbol{x})|^2 \geq z_0^2 \varepsilon^{-1}\big) \\
    &= \int_{z_0^2 \varepsilon^{-1}}^{\infty} \frac{1}{2\sigma^2} e^{-\frac{x}{2\sigma^2}}\, dx \\
    &= \exp\!\left(-\frac{z_0^2}{2\sigma^2}\varepsilon^{-1}\right).
\end{align*}

Taking logarithms and multiplying by \(\varepsilon\), we obtain
\[
    \varepsilon \log \mathbb{P}(\mathscr{A}_\varepsilon^\infty)
    \geq
    -\frac{z_0^2}{2\sigma^2}.
\]
Passing to the limit \(\varepsilon \to 0^+\) yields
\begin{equation*}
    \liminf_{\varepsilon \to 0^+}
    \varepsilon \log \mathbb{P}(\mathscr{A}_\varepsilon^\infty)
    \geq -\frac{z_0^2}{2\sigma^2}.
\end{equation*}

Recalling that \(2\sigma^2 = \sum_{\boldsymbol{n} \in \mathbb{Z}^\nu} |c(\boldsymbol{n})|^2\), we obtain
\begin{equation}\label{eq:lower}
    \liminf_{\varepsilon \to 0^+} \varepsilon \log \mathbb{P}(\mathscr{A}_\varepsilon^\infty)
    \geq -\frac{z_0^2}{\sum_{\boldsymbol{n} \in \mathbb{Z}^\nu} |c(\boldsymbol{n})|^2}.
\end{equation}


\subsection{Upper Bound}\label{sec:upperbound}

In this subsection, we establish an upper bound for the probability of the rogue wave event \(\mathscr{A}_\varepsilon^\infty\). In contrast to the pointwise argument used for the lower bound, the proof of the upper bound relies on a global control of the random field, obtained via polar coordinates, truncation , H\"older's inequality, and standard probabilistic estimates including Chernoff bound. 

\subsubsection{Polar Coordinates of the Random Component}

Let \(g_{\boldsymbol{n}} = r_{\boldsymbol{n}} e^{\iu \theta_{\boldsymbol{n}}}\), where \(r_{\boldsymbol{n}} \ge 0\) and \(\theta_{\boldsymbol{n}}\in[0,2\pi)\) are independent, with
\(2r_{\boldsymbol{n}}^2 \overset{\mathrm{iid}}{\sim} \mathscr{E}(1/2)\) and \(\theta_{\boldsymbol{n}} \overset{\mathrm{iid}}{\sim} \mathscr{U}[0,2\pi]\).
Substituting this polar representation into \eqref{eq:linearsolution}, we obtain
\begin{align}\label{usolution}
u_{\mathrm{linear}}(t,\boldsymbol{x})
=
\sum_{\boldsymbol{n}\in\mathbb{Z}^\nu}
c(\boldsymbol{n})\, r_{\boldsymbol{n}}\, e^{\iu \phi_{\boldsymbol{n}}(t,\boldsymbol{x})},
\end{align}
where the phase
\[
\phi_{\boldsymbol{n}}(t,\boldsymbol{x})
=
\theta_{\boldsymbol{n}}
- Q(\boldsymbol{n})t
+ \langle \boldsymbol{n}\boldsymbol{\Omega}, \boldsymbol{x} \rangle.
\]

To control the supremum of the field, we observe that
\begin{align*}
\sup_{\boldsymbol{x}\in\mathbb{R}^d} |u_{\mathrm{linear}}(t,\boldsymbol{x})|^2
&=
\sup_{\boldsymbol{x}\in\mathbb{R}^d}
\operatorname{Re}\!\left(u_{\mathrm{linear}}(t,\boldsymbol{x})\,\overline{u_{\mathrm{linear}}(t,\boldsymbol{x})}\right) \nonumber\\
&=
\sup_{\boldsymbol{x}\in\mathbb{R}^d}
\sum_{\boldsymbol{n},\boldsymbol{m}\in\mathbb{Z}^\nu}
r_{\boldsymbol{n}} r_{\boldsymbol{m}}
\operatorname{Re}\!\left(
c(\boldsymbol{n})\,\overline{c(\boldsymbol{m})}\,
e^{\iu(\phi_{\boldsymbol{n}}(t,\boldsymbol{x})-\phi_{\boldsymbol{m}}(t,\boldsymbol{x}))}
\right) \nonumber\\
&\le
\sum_{\boldsymbol{n},\boldsymbol{m}\in\mathbb{Z}^\nu}
r_{\boldsymbol{n}} r_{\boldsymbol{m}}\, |c(\boldsymbol{n})|\, |c(\boldsymbol{m})| \nonumber\\
&=
\left(\sum_{\boldsymbol{n}\in\mathbb{Z}^\nu} |c(\boldsymbol{n})|\, r_{\boldsymbol{n}}\right)^2.
\end{align*}

Consequently, we obtain the uniform bound
\begin{align}\label{eq:cr}
\sup_{\boldsymbol{x}\in\mathbb{R}^d} |u_{\mathrm{linear}}(t,\boldsymbol{x})|
\le
\sum_{\boldsymbol{n}\in\mathbb{Z}^\nu} |c(\boldsymbol{n})|\, r_{\boldsymbol{n}},
\end{align}
in which we have gotten rid of the supremum on $\mathbb R^d$. 
\begin{remark}
Set
\[
\mathscr A^{(1)}_{\varepsilon}=\left\{\sum_{\boldsymbol{n}\in\mathbb{Z}^\nu} |c(\boldsymbol{n})|\, r_{\boldsymbol{n}}>z_0\varepsilon^{-1/2}\right\}.
\]
Clearly, 
\begin{align}\label{eq:i1}
\mathscr{A}_\varepsilon^\infty\subset\mathscr A_\varepsilon^{(1)}.
\end{align}
\end{remark}

\subsubsection{Truncation}

To study the probability of $\mathscr A_\varepsilon^{(1)}$, we introduce the truncation parameter
\begin{align}\label{eq:N}
N=\varepsilon^{-\frac{1}{\nu}+\mu_1},
\qquad
0<\mu_1\ll\frac{1}{\nu}.
\end{align}

We define the truncated lattice \(\Lambda_N\subset\mathbb Z^\nu\) by
\begin{align*}
\Lambda_N
&=\left\{\boldsymbol n\in\mathbb Z^\nu:\ |\boldsymbol n|\le N\right\} \\
&=\left\{
(n_1,\ldots,n_d)\in\prod_{j=1}^{d}\mathbb Z^{\nu_j}
:\ |n_j|\le N,\quad j=1,\ldots,d
\right\} \\
&=\left\{
(n_{1,1},\ldots,n_{1,\nu_1};\,\ldots;\,n_{d,1},\ldots,n_{d,\nu_d})
\in\prod_{j=1}^{d}\mathbb Z^{\nu_j}
:\ |n_{j,j'}|\le N,\;
j'=1,\ldots,\nu_j,\;
j=1,\ldots,d
\right\}.
\end{align*}

Using the truncation set \(\Lambda_N\), we decompose the right-hand side of \eqref{eq:cr} into a principal part \(\mathscr S_N\) and a remainder term \(\mathscr R_N\):
\begin{equation}\label{eq:sr}
\sum_{\boldsymbol n\in\mathbb Z^\nu}
|c(\boldsymbol n)|\,r_{\boldsymbol n}
=
\mathscr S_N+\mathscr R_N,
\end{equation}
where
\begin{equation*}
\mathscr S_N
=
\sum_{\boldsymbol n\in\Lambda_N}
|c(\boldsymbol n)|\,r_{\boldsymbol n},
\qquad
\mathscr R_N
=
\sum_{\boldsymbol n\notin\Lambda_N}
|c(\boldsymbol n)|\,r_{\boldsymbol n}.
\end{equation*}

\begin{remark}
With the above notation, the event \(\mathscr A_{\varepsilon}^{(1)}\) can be expressed as
\begin{align}\label{eq:a1}
\mathscr A_{\varepsilon}^{(1)}
=
\left\{
\mathscr S_N+\mathscr R_N
>
z_0\,\varepsilon^{-1/2}
\right\}.
\end{align}
\end{remark}

\subsubsection{Analysis of \(\mathscr R_N\)}

\begin{lemma}[Estimates of \(\mathscr R_N\)]\label{lemma:vanishing}
Let \(N\sim\varepsilon^{-\frac1\nu+\mu_1}\) be defined by \eqref{eq:N}, and let \(\delta\) be as in \autoref{cor:polynomial_decay}. Then, for every \(\omega\notin\Omega_\delta\),
\begin{align}\label{rne}
\mathscr R_N
=
\mathcal O\!\left(
\varepsilon^{-\frac12+\frac{\mu_2}2}
\right).
\end{align}
\end{lemma}

\begin{proof}
Fix \(\omega\notin\Omega_\delta\). By Corollary~\ref{cor:polynomial_decay},
\[
|c(\boldsymbol n)|\,r_{\boldsymbol n}(\omega)
=
|c(\boldsymbol n)g_{\boldsymbol n}(\omega)|
\le
\delta^{-1/2}
\tnorm{\boldsymbol n}_{-(\rho-\kappa)}.
\]
Hence
\begin{align*}
\mathscr R_N
&=
\sum_{|\boldsymbol n|>N}
|c(\boldsymbol n)|\,r_{\boldsymbol n}(\omega) \\
&\le
\delta^{-1/2}
\sum_{|\boldsymbol n|>N}
\tnorm{\boldsymbol n}_{-(\rho-\kappa)} \\
&\lesssim
\delta^{-1/2}
\prod_{j=1}^{d}\prod_{j'=1}^{\nu_j}
\sum_{|n_{j,j'}|>N}
(1+|n_{j,j'}|)^{-(\rho_{j,j'}-\kappa_{j,j'})}.
\end{align*}

Since \(\rho_{j,j'}-\kappa_{j,j'}>2\), 
\[
\sum_{|n_{j,j'}|>N}
(1+|n_{j,j'}|)^{-(\rho_{j,j'}-\kappa_{j,j'})}
\lesssim
N^{-(\rho_{j,j'}-\kappa_{j,j'}-1)}.
\]
Therefore
\[
\mathscr R_N
\lesssim
\delta^{-1/2}
N^{-\eta_2},
\]
where $\eta_2$ is defined by \eqref{eta2}. 

Using \(N\sim\varepsilon^{-\frac1\nu+\mu_1}\), we obtain
\[
\mathscr R_N
\lesssim
\delta^{-1/2}
\varepsilon^{\left(\frac1\nu-\mu_1\right)
\eta_2}.
\]

By assumption \eqref{eq:delta} on \(\delta\),
\[
\delta^{-1/2}
\le
\varepsilon^{-\frac12
-\left(\frac1\nu-\mu_1\right)
\eta_2
+\frac{\mu_2}2}.
\]
Combining the above estimates, we conclude that
\[
\mathscr R_N
\lesssim
\varepsilon^{-\frac12+\frac{\mu_2}2},
\]
which proves the claim.
\end{proof}

\begin{remark}
By Lemma~\ref{lemma:vanishing}, there exists a constant \(C>0\) such that
\[
\mathscr R_N
\le
C\,\varepsilon^{-\frac12+\frac{\mu_2}2}.
\]
Define
\[
\widetilde{\Omega}
=
\left\{
\mathscr R_N
\le
C\,\varepsilon^{-\frac12+\frac{\mu_2}2}
\right\},
\]
and
\[
\mathscr A^{(2)}_\varepsilon
=
\left\{
\mathscr S_N
>
\bigl(z_0-C\varepsilon^{\frac{\mu_2}{2}}\bigr)\varepsilon^{-\frac{1}{2}}
\right\}.
\]
Then
\begin{align}\label{eq:i2}
\mathscr A_{\varepsilon}^{(1)}
\cap
\widetilde{\Omega}
\subset
\mathscr A_\varepsilon^{(2)},
\end{align}
and
\begin{align}\label{eq:io}
\Omega_\delta^{\,c}
\subset
\widetilde{\Omega}.
\end{align}
\end{remark}
\subsubsection{Analysis of $\mathscr S_N$.}
Here we analyze the principal part \(\mathscr{S}_N\), which provides the dominant contribution to the upper bound. The analysis reduces to studying a sum of independent random variables after separating deterministic and random components.

Clearly, 
\[
\mathscr A^{(2)}_\varepsilon=\left\{\mathscr S_N^2>(z_0-C\varepsilon^{\frac{\mu_2}{2}})^2
\varepsilon^{-1}\right\}. 
\]
Applying H\"older's inequality, we obtain
\begin{align*}
\mathscr S_N^2
&\leq \sum_{\boldsymbol n\in\Lambda_N}\frac{1}{2}|c(\boldsymbol n)|^2 \cdot \sum_{\boldsymbol n\in\Lambda_N}2r_{\boldsymbol n}^2 \nonumber\\
&\leq \left(\sum_{\boldsymbol n\in\mathbb Z^\nu}\frac{1}{2}|c(\boldsymbol n)|^2\right)\cdot \left(\sum_{\boldsymbol n\in\Lambda_N}2r_{\boldsymbol n}^2\right).
\end{align*}
\begin{remark}
Set
\begin{align*}
\mathscr A_\varepsilon^{(3)}&=
\left\{\xi_N>\mathfrak{I}_\varepsilon\varepsilon^{-1}\right\},
\end{align*}
where
\begin{align}
\xi_N&=\sum_{\boldsymbol n\in\Lambda_N}2r_{\boldsymbol n}^2,\nonumber\\
\mathfrak{I}_\varepsilon
&= \frac{\big(z_0 - C\varepsilon^{\frac{\mu_2}{2}}\big)^2}{\frac{1}{2}\sum_{\boldsymbol{n}\in\mathbb{Z}^\nu} |c(\boldsymbol{n})|^2}.\label{rate}
\end{align}
Clearly, 
\begin{align}\label{eq:i3}
\mathscr A^{(2)}_\varepsilon\subset\mathscr A^{(3)}_{\varepsilon}.
\end{align}
\end{remark}

\begin{lemma}[LDP for $\mathscr A^{(3)}_{\varepsilon}$]\label{ldp3}
Let \(N \sim \varepsilon^{-\frac1\nu+\mu_1}\) be defined in \eqref{eq:N}, we have
\[
\lim_{\varepsilon\to 0^+}\varepsilon\log\mathbb P(\mathscr A^{(3)}_{\varepsilon})=-\frac{z_0^2}{\sum_{\boldsymbol{n}\in\mathbb{Z}^\nu} |c(\boldsymbol{n})|^2}.
\]
\end{lemma}
\begin{proof}
{\bf Step 1. Lower bound via pointwise control}

Clearly, $\mathscr A_\varepsilon^{(3)}\supset\{2r_{\boldsymbol{n}}^2 > \mathfrak{I}_\varepsilon \varepsilon^{-1}\}$, where $2r_{\boldsymbol{n}}^2\sim\mathscr{E}(1/2)$. Hence we have
\begin{align*}
\mathbb{P}(\mathscr A_\varepsilon^{(3)})
&\geq \mathbb P(2r_{\boldsymbol{n}}^2 > \mathfrak{I}_\varepsilon \varepsilon^{-1}) \\
&= \int_{\mathfrak{I}_\varepsilon \varepsilon^{-1}}^{\infty} \frac{1}{2} e^{-\frac{1}{2}x} \, dx \\
&= e^{-\frac{1}{2} \mathfrak{I}_\varepsilon \varepsilon^{-1}}.
\end{align*}
Taking logarithms, multiplying by \(\varepsilon\), and then taking the lower limit with respect to $\varepsilon\to0^+$, we obtain that 
\[
\liminf_{\varepsilon\to0^+}\varepsilon\log\mathbb{P}(\mathscr A_\varepsilon^{(3)})\geq\liminf_{\varepsilon\to0^+}\left(-\frac{1}{2}\mathfrak{I}_\varepsilon\right)
=-\frac{z_0^2}{\sum_{\boldsymbol{n}\in\mathbb Z^\nu}|c(\boldsymbol{n})|^2}.\]

{\bf Step 2. Upper bound via Chernoff bound}

Since \(2 r_{\boldsymbol n}^2 \sim \chi^2_2 = \mathscr E(1/2)\), its moment generating function is given by
\[
\mathbb{E}\bigl[e^{\lambda \cdot 2 r_{\boldsymbol n}^2}\bigr] = (1 - 2\lambda)^{-1}, \qquad \lambda \in (0, \tfrac12).
\]
By independence and \(|\Lambda_N| = (2N+1)^\nu\), it follows that
\[
\mathbb{E}\bigl[e^{\lambda \xi_N}\bigr] = (1 - 2\lambda)^{-(2N+1)^\nu}.
\]

Applying Lemma~\ref{lemma:ccbound}, we obtain
\[
\mathbb{P}(\mathscr A_\varepsilon^{(3)}) \le (1 - 2\lambda)^{-(2N+1)^\nu} e^{-\lambda \mathfrak{I}_\varepsilon\varepsilon^{-1}}.
\]

To optimize over \(\lambda\in(0,1/2)\), set
\[
\log h(\lambda) = -(2N+1)^\nu \log(1 - 2\lambda) - \lambda \mathfrak{I}_\varepsilon\varepsilon^{-1}.
\]
Then
\[
\frac{d}{d\lambda} \log h(\lambda) = \frac{2(2N+1)^\nu}{1 - 2\lambda} - \mathfrak{I}_\varepsilon\varepsilon^{-1}.
\]
Solving \(\frac{d}{d\lambda} \log h(\lambda)=0\) yields
\[
\lambda^* = \frac{1}{2}\left(1 - \frac{2(2N+1)^\nu}{\mathfrak{I}_\varepsilon\varepsilon^{-1}}\right),
\]
which lies in \((0,1/2)\) provided that \(\mathfrak{I}_\varepsilon\varepsilon^{-1} > 2(2N+1)^\nu\). This can be achieved by choosing an appropriate scaling of \(N\) with respect to \(\varepsilon\); see \eqref{eq:N}. 

Thus we have
\[
\mathbb{P}(\mathscr A_\varepsilon^{(3)}) \le (1 - 2\lambda^*)^{-(2N+1)^\nu} e^{-\lambda^* \mathfrak{I}_\varepsilon\varepsilon^{-1}}.
\]
Taking logarithms and multiplying both sides by \(\varepsilon\) gives
\[
\varepsilon \log \mathbb{P}(\mathscr A_\varepsilon^{(3)}) 
\le -\lambda^*\mathfrak{I}_\varepsilon - \varepsilon(2N+1)^\nu \log(1-2\lambda^*).
\]
Substituting \(\lambda^*\) and simplifying,
\[
\varepsilon \log \mathbb{P}(\mathscr A_\varepsilon^{(3)}) 
\le -\frac{1}{2}\mathfrak{I}_\varepsilon + \varepsilon(2N+1)^\nu - \varepsilon(2N+1)^\nu \log\!\left(\frac{2 \varepsilon (2N+1)^\nu}{\mathfrak{I}_\varepsilon}\right).
\]

As \(\varepsilon \to 0^+\), recalling that \(N \sim \varepsilon^{-\frac{1}{\nu}+\mu}\) with \(0<\mu<1/\nu\) as in \eqref{eq:N}, we have
\[
\mathfrak{I}_\varepsilon \to \frac{z_0^2}{\frac{1}{2}\sum_{\boldsymbol{n}\in\mathbb{Z}^\nu}|c(\boldsymbol{n})|^2}, \quad
 \varepsilon(2N+1)^\nu\to 0, \quad \text{and} \quad
 \varepsilon (2N+1)^\nu\log\!\left(\frac{2 \varepsilon (2N+1)^\nu }{\mathfrak{I}_\varepsilon}\right)\to 0.
\]
These imply that
\[
\limsup_{\varepsilon\to0^+} \varepsilon \log \mathbb{P}(\mathscr A_\varepsilon^{(3)}) \le -\frac{z_0^2}{\sum_{\boldsymbol{n}\in\mathbb Z^\nu}|c(\boldsymbol{n})|^2}.
\] 
Proof of \autoref{ldp3} is complete.
\end{proof}


\subsubsection{Putting Everything Together}\label{sec:all}
It follows from \eqref{eq:i1}, \eqref{eq:a1}, \eqref{eq:i2}, \eqref{eq:io}, \eqref{eq:i3} and the law of total probability that
\begin{align*}
\mathbb P(\mathscr A_\varepsilon^\infty)\le\mathbb P(\mathscr A_\varepsilon^{(1)})
&=\mathbb P(\mathscr A_\varepsilon^{(1)}\cap\widetilde\Omega)+
\mathbb P(\mathscr A_\varepsilon^{(1)}\cap\widetilde\Omega^c)\\
&\le\mathbb P(\mathscr A_\varepsilon^{(2)})+\mathbb P(\widetilde\Omega^c)\\
&\le\mathbb P(\mathscr A_\varepsilon^{(3)})+\mathbb P(\Omega_\delta)\\
&=\mathbb P(\mathscr A_\varepsilon^{(3)})\left(1+\frac{\mathbb P(\Omega_\delta)}{\mathbb P(\mathscr A_\varepsilon^{(3)})}\right).
\end{align*}

Taking logarithms and multiplying by \(\varepsilon\), we have
\begin{align}\label{ee}
\varepsilon\log\mathbb{P}(\mathscr A_\varepsilon^\infty)\leq\varepsilon\log\mathbb{P}(\mathscr A^{(3)}_\varepsilon)+\varepsilon\log\left(1+\frac{\mathbb P(\Omega_\delta)}{\mathbb P(\mathscr A_\varepsilon^{(3)})}\right).
\end{align}
The first term satisfies a large deviation principle (see \autoref{ldp3}), namely
\[
\mathbb P(\mathscr A_\varepsilon^{(3)}) = \exp\!\left(-\frac{\mathfrak I_0}{2\varepsilon} + o(\varepsilon^{-1})\right).
\]
The second term is controlled by \autoref{thm:poly}, namely
\[
\mathbb P(\Omega_\delta) \lesssim \exp(-\delta^{-1}),\quad \delta\sim\varepsilon^{1+\eta}<\varepsilon,
\]
where \(\eta\) is specified in \autoref{cor:polynomial_decay}. Thus, 
\[
\lim_{\varepsilon \to 0^+} \varepsilon \log\!\left(1 + \frac{\mathbb P(\Omega_\delta)}{\mathbb P(\mathscr A_\varepsilon^{(3)})}\right) = 0.
\]

Taking the upper limit with respect to $\varepsilon\to0^+$ on both sides of \eqref{ee}, we have 
\begin{align*}
\limsup_{\varepsilon \to 0^+} \varepsilon \log \mathbb{P}(\mathscr A_\varepsilon^\infty)
    &\le -\frac{z_0^2}{\sum_{\boldsymbol{n}\in\mathbb{Z}^\nu} |c(\boldsymbol{n})|^2},
\end{align*}
which agrees with the right-hand side of \eqref{eq:lower}. This completes the proof of \autoref{thm:linearldp}.

\section{Well-Posedness}\label{sec:wellp}

In \autoref{sec:linearldp}, we established the linear LDP, which serves as a baseline for the nonlinear analysis. Before turning to the nonlinear case, we first derive in \autoref{sec:wellp} a suitable well-posedness result that captures the relevant time scale, along with the necessary estimates.

\begin{theorem}[Existence, Uniqueness and Fourier Decay]\label{thm:eude}

Consider the randomized quasi-periodic Cauchy problem \eqref{eq:nls}--\eqref{eq:initialdata} together with \autoref{cor:polynomial_decay}. 
Then there exists a positive time
\[
{
{T_\varepsilon = \mathcal O\left(\varepsilon^{-\alpha+1+\eta}\right)},
}
\]
where $\eta$ is chosen in accordance with \eqref{eq:delta}, and a randomized spatially quasi-periodic solution of the form
\begin{align}\label{eq:solution2}
{u}(t,\boldsymbol{x})
= \sum_{\boldsymbol{n}\in\mathbb{Z}^\nu}
{c}(t,\boldsymbol{n})
e^{\iu\langle \boldsymbol{n}\boldsymbol{\Omega}, \boldsymbol{x}\rangle}
\end{align}
defined on $[0,T_\varepsilon]$, with Fourier coefficients satisfying the uniform-in-time decay estimates
\begin{align}\label{eq:decay2}
\sup_{t\in[0,T_\varepsilon]}
|{c}(t,\boldsymbol{n})|
\lesssim {\varepsilon^{-\frac{1}{2}-\frac{\eta}{2}}}\tnorm{\boldsymbol{n}}_{-\frac{\rho-\kappa}{2}}.
\end{align}

\end{theorem}

The proof of \autoref{thm:eude} is based on the strategy developed in \cite{C07,DG16,X25}. 
We outline the main steps, emphasizing the key modifications.

\subsection{The Lattice System and Picard Iteration}

It follows from the orthogonality of 
$\{e^{\mathrm{i}\langle\boldsymbol{n}\Omega,\boldsymbol{x}\rangle}\}_{\boldsymbol{n}\in\mathbb Z^\nu}$
(see \cite[Lemma 1.1]{X25}) that, in Fourier space, \eqref{eq:nls} reads
\begin{align}\label{eq:fourier_nls}
\partial_t{c}(t,\boldsymbol{n}) 
= 
-\mathrm{i} Q(\boldsymbol{n}){c}(t,\boldsymbol{n})
+ \mathrm{i}{\varepsilon^\alpha} 
\sum_{\substack{\boldsymbol{n}^j \in \mathbb Z^\nu,\, j=1,2,3\\ 
\boldsymbol{n}^1 - \boldsymbol{n}^2 + \boldsymbol{n}^3 = \boldsymbol{n}}}
\prod_{j=1}^3 \bigl\{{c}(t,\boldsymbol{n}^j)\bigr\}^{\ast^{[j-1]}}.
\end{align}

The initial condition in Fourier space is given by
\begin{align}\label{eq:initialdata_fourier}
{c}(0,\boldsymbol{n})={c}(\boldsymbol{n})g_{\boldsymbol{n}}.
\end{align}

For each $\boldsymbol{n} \in \mathbb{Z}^\nu$, the Duhamel formulation reads
\begin{align}\label{eq:duhamel_c}
{c}(t,\boldsymbol{n})
= e^{-\iu Q(\boldsymbol{n})t}{c}(\boldsymbol{n})g_{\boldsymbol{n}}
+ \iu{\varepsilon^\alpha}  \int_0^t e^{-\iu Q(\boldsymbol{n})(t-s)}
\sum_{\substack{\boldsymbol{n}^j \in \mathbb Z^\nu,\, j=1,2,3\\ 
\boldsymbol{n}^1 - \boldsymbol{n}^2 + \boldsymbol{n}^3 = \boldsymbol{n}}}
\prod_{j=1}^{3}
\bigl\{{c}(s,\boldsymbol{n}^j)\bigr\}^{\ast^{[j-1]}}
\,\mathrm{d}s.
\end{align}

Define the Picard iteration $\{{c}_k(t,\boldsymbol{n})\}_{k\ge0}$ by
\begin{align}\label{eq:initialguess}
{c}_0(t,\boldsymbol{n})
=
e^{-\iu Q(\boldsymbol{n})t}{c}(\boldsymbol{n})g_{\boldsymbol{n}},
\end{align}
and for $k\ge1$,
\begin{align}\label{eq:picard}
{c}_k(t,\boldsymbol{n})
=
{c}_0(t,\boldsymbol{n})
+ \iu{\varepsilon^\alpha}  \int_0^t e^{-\iu Q(\boldsymbol{n})(t-s)}
\sum_{\substack{\boldsymbol{n}^j \in \mathbb Z^\nu,\, j=1,2,3\\ 
\boldsymbol{n}^1 - \boldsymbol{n}^2 + \boldsymbol{n}^3 = \boldsymbol{n}}}
\prod_{j=1}^{3}
\bigl\{{c}_{k-1}(s,\boldsymbol{n}^j)\bigr\}^{\ast^{[j-1]}}
\,\mathrm{d}s.
\end{align}

\subsection{Tree Structure of the Picard Sequence}

In this subsection, we represent the Picard iteration using tree expansions, 
in order to systematically encode the combinatorial structure of the Duhamel terms. 
This representation allows us to track the multilinear interactions arising from the nonlinear evolution; see \autoref{prop:tree}.

We begin by illustrating the decomposition of the linear and nonlinear contributions 
(see \autoref{fig:diagramlinear} and \autoref{fig:diagramduhamel}). We then briefly recall the combinatorial formalism developed in \cite{DG16,DLX24,X25}, which provides the basis for the tree representation.
\begin{figure}[htpb!]
\centering
\begin{tikzpicture}[scale=0.9,transform shape]
\node at (-2.7,2.6) {$\mathbb R^d\ni\boldsymbol{x}=(x_1,\cdots,x_d)$};
\node at (-0.5,2.5) {$\mapsto$};
\node at (1.65,2.5) {$\mathbb Z^\nu=\mathbb Z^{\nu_1}\times\cdots\times\mathbb Z^{\nu_d}$};
\node at (0,2) {\rotatebox{90}{$\in$}};
\node at (1.29,1.5) {$\boldsymbol{n}=(n_{1},\cdots,n_d)$};
\node at (0,1) {\color{green}$\bullet$};
\draw [dashed] (-0.5,1) -- (1.1,1);
\node at (1.75,1) {${c}_k(t,\boldsymbol{n})$};
\draw (0,0) -- (0,1);
\node at (0,0) {$\bullet$};
\draw [dashed] (-3.5,0) -- (1.1,0);
\node at (0,-0.35) {$(n_{1},\cdots,n_d)$};
\node at (1.75,0) {${c}_0(t,\boldsymbol{n})$};

\node at (-5.3,0.04) {$\mathrm D^{(k,0)}\ni\spadesuit^{(k)}$};
\node at (-3.9,0.04) {\rotatebox{180}{$\mapsto$}};

\node at (-2.5,-1.5) {$\blacksquare$};
\draw (-0.6,-0.6)--(-2.5,-1.5);

\node at (2.5,-1.5) {$\blacksquare$};
\draw (0.6,-0.6)--(2.5,-1.5);

\node at (-1.6,-2) {$\widetilde\spadesuit_{x_1}^{(k)}\in\mathscr D_{x_1}^{(k,0)}$};

\node at (3.42,-2) {$\widetilde\spadesuit_{x_d}^{(k)}\in\mathscr D_{x_d}^{(k,0)}$};

\node at (0,-3.5) {$\blacksquare$};
\draw (-2.5,-2.25)--(0,-3.5);
\draw (2.5,-2.25)--(0,-3.5);

\node at (-4.55,-1.9) {$\iota_{x_1}$};
\node at (-4.7,-3.9) {$\iota$};
\node at (-4.5,-5.1) {$\iota_{x_d}$};

\node at (0.92,-4) {$\widetilde\spadesuit^{(k)}\in\mathscr D^{(k,0)}$};

\coordinate (A) at (-4.8,-0.2);
\coordinate (B) at (-4.8,-2.1);
\coordinate (C) at (-2.7,-2.1);

\coordinate (D) at (-4.8,-4.05);
\coordinate (E) at (-0.2,-4.05);

\coordinate (F) at (-4.8,-5.3);
\coordinate (G) at (2.52,-5.3);
\coordinate (H) at (2.52,-2.3);

\draw [blue] (A)--(B);
\draw [->,blue] (B)--(C);
\draw [blue] (A)--(D);
\draw [->,blue] (D)--(E);
\draw [blue] (A)--(F);
\draw [blue] (F)--(G);
\draw [->,blue] (G)--(H);

\end{tikzpicture}
\caption{Diagram of the linear component ${\mathfrak c}_0(t,\boldsymbol{n})$. 
}
\label{fig:diagramlinear}
\end{figure}

\begin{figure}[htpb!]
\centering
\begin{tikzpicture}[scale=0.7,transform shape]

\node at (5.5,9.3) {$\mathbb R^d\ni\boldsymbol{x}$};
\node at (6.5,9.2) {$\mapsto$};
\node at (7.1,9.2) {$\mathbb{Z}^{\nu}$};
\node at (7,8.8) {\rotatebox{90}{$\in$}};
\node at (7,8.4) {$\boldsymbol{n}$};
\node at (7,8) {\color{green}$\bullet$};
\draw [dashed] (6.5,8)--(7.8,8);
\node at (8.5,8) {${c}_k(t,\boldsymbol{n})$};

\node at (6.8,7.2) {$\boldsymbol{n}$};
\node at (7,7) {\color{red}$\bullet$};
\draw [dashed] (6.5,7)--(7.8,7);
\node at (9,7) {$({c}_k-{c}_0)(t,\boldsymbol{n})$};

\node at (-0.21,6.08) {$\mathrm{D}^{(k,\gamma^{(k)})}\ni\spadesuit^{(k)}$};
\node at (1.2,6.02) {\rotatebox{-180}{$\mapsto$}};
\node at (2,6) {$\bullet$};
\node at (2.8,5.7) {$\boldsymbol{n}^1\in\mathbb{Z}^{\nu}$};
\node at (7,6) {\color{red}$\bullet$};
\node at (7.8,5.7) {$\boldsymbol{n}^2\in\mathbb{Z}^{\nu}$};
\node at (12,6) {$\bullet$};
\node at (12.8,5.7) {$\boldsymbol{n}^3\in\mathbb{Z}^{\nu}$};
\draw [dashed] (1.5,6)--(12.5,6);

\node at (2,5) {$\bullet$};
\node at (3.8,4.65) {$\spadesuit^{(k-1)}_1\in\mathrm{D}^{(k-1,\gamma_1^{(k-1)})}$};
\node at (7,5) {\color{red}$\bullet$};
\node at (8.8,4.65) {$\spadesuit^{(k-1)}_2\in\mathrm{D}^{(k-1,\gamma_2^{(k-1)})}$};
\node at (12,5) {$\bullet$};
\node at (13.8,4.65) {$\spadesuit^{(k-1)}_3\in\mathrm{D}^{(k-1,\gamma_3^{(k-1)})}$};
\draw [dashed] (1.5,5)--(12.5,5);

\node at (2,4) {$\bullet$};
\node at (2.75,3.7) {$(\mathbb{Z}^{\nu})^{2\sigma(\gamma_1^{(k-1)})}$};
\node at (2,3.2) {\rotatebox{90}{$\in$}};
\node at (3,2.5) {$(n^{1,j})_{1\le j\le 2\sigma(\gamma_1^{(k-1)})}$};
\node at (2,2) {\rotatebox{90}{$=$}};
\node at (3.5,1.3) {$\bigl((n_1^{1,j},\cdots,n_d^{1,j})\bigr)_{1\le j\le 2\sigma(\gamma_1^{(k-1)})}$};

\node at (7,4) {\color{red}$\bullet$};
\node at (7.75,3.7) {$(\mathbb{Z}^{\nu})^{2\sigma(\gamma_2^{(k-1)})}$};
\node at (7,3.2) {\rotatebox{90}{$\in$}};
\node at (8,2.5) {$(n^{2,j})_{1\le j\le 2\sigma(\gamma_1^{(k-1)})}$};
\node at (7,2) {\rotatebox{90}{$=$}};
\node at (8.5,1.3) {$\bigl((n_1^{2,j},\cdots,n_d^{2,j})\bigr)_{1\le j\le 2\sigma(\gamma_2^{(k-1)})}$};

\node at (12,4) {$\bullet$};
\node at (12.75,3.7) {$(\mathbb{Z}^{\nu})^{2\sigma(\gamma_3^{(k-1)})}$};
\node at (12,3.2) {\rotatebox{90}{$\in$}};
\node at (13,2.5) {$(n^{2,j})_{1\le j\le 2\sigma(\gamma_1^{(k-1)})}$};
\node at (12,2) {\rotatebox{90}{$=$}};
\node at (14,1.3) {$\bigl((n_1^{3,j},\cdots,n_d^{3,j})\bigr)_{1\le j\le 2\sigma(\gamma_3^{(k-1)})}$};

\draw [dashed] (1.5,4)--(12.5,4);

\draw (7,8)--(7,7);
\draw (7,7)--(2,6);
\draw (7,7)--(7,6);
\draw (7,7)--(12,6);
\draw (2,5)--(2,6);
\draw (2,5)--(2,4);
\draw (7,5)--(7,6);
\draw (7,5)--(7,4);
\draw (12,5)--(12,6);
\draw (12,5)--(12,4);

\node at (3.25,-1) {$\blacksquare$};
\draw (1.7,1)--(3.25,-1);

\node at (4.5,-1) {\color{red}$\blacksquare$};
\draw [red](6.5,1)--(4.5,-1);

\node at (5.75,-1) {$\blacksquare$};
\draw (12.1,1)--(5.75,-1);

\node at (8.25,-1) {$\blacksquare$};
\draw (3.2,1)--(8.25,-1);

\node at(9.5,-1) {\color{red}$\blacksquare$};
\draw [red](8.2,1)--(9.5,-1);

\node at (10.75,-1) {$\blacksquare$};
\draw (13.5,1)--(10.75,-1);

\node at (7,-1) {$\cdots$};

\node at (3.63,-1.6) {$\widetilde\spadesuit_{1,x_1}^{(k-1)}$};
\node at (4.88,-1.6) {$\widetilde\spadesuit_{2,x_1}^{(k-1)}$};
\node at (6.13,-1.6) {$\widetilde\spadesuit_{3,x_1}^{(k-1)}$};

\node at (8.63,-1.6) {$\widetilde\spadesuit_{1,x_d}^{(k-1)}$};
\node at (9.88,-1.6) {$\widetilde\spadesuit_{2,x_d}^{(k-1)}$};
\node at (11.13,-1.6) {$\widetilde\spadesuit_{3,x_d}^{(k-1)}$};

\node at (4.5,-3.5) {$\blacksquare$};
\node at (5.57,-4) {$\widetilde\spadesuit_{x_1}^{(k)}\in\mathscr{D}_{x_1}^{(k,\gamma^{(k)})}$};
\node at (9.5,-3.5) {$\blacksquare$};
\node at (10.57,-4) {$\widetilde\spadesuit_{x_d}^{(k)}\in\mathscr{D}_{x_d}^{(k,\gamma^{(k)})}$};

\draw (3.27,-2)--(4.5,-3.5);
\draw [red](4.51,-2)--(4.5,-3.5);
\draw (5.75,-2)--(4.5,-3.5);

\draw (8.26,-2)--(9.5,-3.5);
\draw [red](9.51,-2)--(9.5,-3.5);
\draw (10.78,-2)--(9.5,-3.5);

\node at (7,-6) {$\blacksquare$};
\node at (8.08,-6.5) {$\spadesuit^{(k)}\in\mathscr{D}^{(k,\gamma^{(k)})}$};

\draw (7,-6)--(4.5,-4.3);
\draw (7,-6)--(9.44,-4.3);

\coordinate (A) at (0.46,5.7);

\coordinate (B) at (0.46,-4.1);
\draw [blue](A)--(B);

\coordinate (D) at (4.3,-4.1);
\draw [->,blue](B)--(D);

\coordinate (E) at (0.46,-6.6);
\draw [blue](A)--(E);

\coordinate (F) at (6.8,-6.6);
\draw [->,blue](E)--(F);

\coordinate (G) at (0.46,-7.6);
\draw [blue](A)--(G);

\coordinate (H) at (9.48,-7.6);
\draw [blue](G)--(H);

\coordinate (I) at (9.48,-4.4);
\draw [->,blue](H)--(I);

\node at (0.75,-3.9) {$\iota_{x_1}$};
\node at (0.6,-6.4) {$\iota$};
\node at (0.75,-7.4) {$\iota_{x_d}$};

\end{tikzpicture}
\caption{Diagram of the Duhamel term $({c}_k-{\mathfrak c}_0)(t,\boldsymbol{n})$. 
}
\label{fig:diagramduhamel}
\end{figure}

\begin{definition}
The {\em branch set} $\Gamma^{(k)}$ is defined recursively by
\begin{align*}
\Gamma^{(k)} =
\begin{cases}
\{0,1\}, & k = 1; \\
\{0\} \cup \bigl(\Gamma^{(k-1)}\bigr)^3, & k \ge 2,
\end{cases}
\end{align*}
where $0$ corresponds to a linear node and $1$ corresponds to a cubic interaction node.
\end{definition}

\begin{definition}
The {\em first counting function} $\sigma$ on $\Gamma^{(k)}$ is defined by
\begin{align*}
\sigma(\gamma^{(k)})=
\begin{cases}
\frac{1}{2}, & \gamma^{(k)}=0, \ k\ge1; \\
\frac{3}{2}, & \gamma^{(1)}=1; \\
\sum_{j=1}^{3}\sigma(\gamma_j^{(k-1)}), & \gamma^{(k)}=(\gamma_j^{(k-1)})_{j=1}^3,\ k\ge2.
\end{cases}
\end{align*}
\end{definition}

\begin{definition}
The {\em second counting function} $\ell$ is defined by
\begin{align*}
\ell(\gamma^{(k)})=
\begin{cases}
0, & \gamma^{(k)}=0, \ k\ge1; \\
1, & \gamma^{(1)}=1; \\
1+\sum_{j=1}^{3}\ell(\gamma_j^{(k-1)}), & \gamma^{(k)}=(\gamma_j^{(k-1)})_{j=1}^3,\ k\ge2.
\end{cases}
\end{align*}
\end{definition}

\begin{proposition}\label{prop:sigma_ell_identity}
For all $k\ge1$, we have
\begin{align*}
\sigma(\gamma^{(k)})=\ell(\gamma^{(k)})+\frac{1}{2}.
\end{align*}
\end{proposition}

\begin{definition}[Domain]
Define recursively
\begin{align*}
{\mathrm D}^{(k,\gamma^{(k)})}=
\begin{cases}
(\mathbb Z^{\nu}) , & \gamma^{(k)}=0;\\
(\mathbb Z^{\nu})^3, & \gamma^{(1)}=1;\\
\prod_{j=1}^3 {\mathrm D}^{(k-1,\gamma_j^{(k-1)})}, & \gamma^{(k)}\in(\Gamma^{(k-1)})^3.
\end{cases}
\end{align*}
\end{definition}

\begin{proposition}\label{prop:d}
Let $\dim_{\mathbb Z^\nu}\mathrm D^{(k,\gamma^{(k)})}$ denote the number of $\mathbb Z^\nu$-components of $\mathrm D^{(k,\gamma^{(k)})}$. Then
\[
\dim_{\mathbb Z^{\nu}} \mathrm D^{(k,\gamma^{(k)})} = 2\sigma(\gamma^{(k)}).
\]
That is, 
\begin{align*}
\mathrm D^{(k,\gamma^{(k)})} \cong (\mathbb Z^{\nu})^{2\sigma(\gamma^{(k)})}.
\end{align*}
\end{proposition}

\begin{proposition}
The domain ${\mathrm D}^{(k,\gamma^{(k)})}$ admits a canonical identification with
\begin{align*}
{\mathscr D}^{(k,\gamma^{(k)})}=
\begin{cases}
\prod_{j=1}^d \mathbb{Z}^{\nu_j}, & \gamma^{(k)}=0,\ k\ge1;\\
\prod_{j=1}^d (\mathbb{Z}^{\nu_j})^3, & \gamma^{(1)}=1;\\
\prod_{j=1}^3 {\mathscr D}^{(k-1,\gamma_j^{(k-1)})}, & \gamma^{(k)}\in(\Gamma^{(k-1)})^3,\ k\ge2.
\end{cases}
\end{align*}
This identification is induced by coordinate-wise projection.

More precisely, for each $j=1,\dots,d$, we define the projection
\[
\iota_{x_j}:\left(\prod_{j=1}^d \mathbb{Z}^{\nu_j}\right)^s \to (\mathbb{Z}^{\nu_j})^s,
\quad
\iota_{x_j}\big((q_1^1,\dots,q_d^1),\dots,(q_1^s,\dots,q_d^s)\big)
= (q_j^1,\dots,q_j^s).
\]
The full map $\iota=(\iota_{x_1},\dots,\iota_{x_d})$ induces the above identification.
\end{proposition}

\begin{definition}[Alternating sum]
We define the map
\[
\lambda = \cas \circ \iota
= (\cas_{x_1}\circ \iota_{x_1}, \dots, \cas_{x_d}\circ \iota_{x_d}),
\]
where for each $j=1,\dots,d$, the map
\[
\cas_{x_j}: \bigcup_{s=1}^\infty (\mathbb Z^{\nu_j})^s \to \mathbb Z^{\nu_j}
\]
is defined by
\[
\cas_{x_j}(q_1,\dots,q_s)
:= \sum_{m=1}^s (-1)^{m-1} q_m.
\]
\end{definition}

\begin{proposition}[Tree representation of ${c}_{k}(t,\boldsymbol{n})$]\label{prop:tree}

For all $k\geq1$, we have
\[
{c}_k(t,\boldsymbol{n})
=
\sum_{\gamma^{(k)}\in\Gamma^{(k)}}
\sum_{\substack{\spadesuit^{(k)}\in{\mathrm D}^{(k,\gamma^{(k)})}\\
\lambda(\spadesuit^{(k)})=\boldsymbol{n}}}
\mathfrak C^{(k,\gamma^{(k)})}(\spadesuit^{(k)})
\mathfrak I^{(k,\gamma^{(k)})}(t,\spadesuit^{(k)})
\mathfrak F^{(k,\gamma^{(k)})}(\spadesuit^{(k)}),
\]
where for $0\in\Gamma^{(k)}$,
\begin{align*}
\mathfrak C^{(k,0)}(\spadesuit^{(k)})
&={c}(\lambda(\spadesuit^{(k)}))g_{\lambda(\spadesuit^{(k)})},\\
\mathfrak I^{(k,0)}(t,\spadesuit^{(k)})
&=e^{-\iu Q(\lambda(\spadesuit^{(k)})) t},\\
\mathfrak F^{(k,0)}(\spadesuit^{(k)})
&=1;
\end{align*}
for $1\in\Gamma^{(1)}$,
\begin{align*}
\mathfrak C^{(1,1)}(\spadesuit^{(1)})
&=\prod_{j=1}^3\{{c}(\boldsymbol{n}^j)g_{\boldsymbol{n}^j}\}^{\ast^{[j-1]}},\\
\mathfrak I^{(1,1)}(t,\spadesuit^{(1)})
&=\int_0^t e^{-\mathrm{i} Q(\lambda(\spadesuit^{(1)}))(t-s)}
\prod_{j=1}^{3}
\left\{e^{-\mathrm{i} Q(\boldsymbol{n}^j) s}\right\}^{\ast^{[j-1]}}
\,\mathrm{d}s,\\
\mathfrak F^{(1,1)}(\spadesuit^{(1)})
&=\mathrm{i}{\varepsilon^\alpha} ;
\end{align*}
and for $\gamma^{(k)}=(\gamma_j^{(k-1)})_{j=1}^3$,
\begin{align*}
\mathfrak C^{(k,\gamma^{(k)})}(\spadesuit^{(k)})
&=\prod_{j=1}^3
\left\{\mathfrak C^{(k-1,\gamma_j^{(k-1)})}(\spadesuit_j^{(k-1)})\right\}^{\ast^{[j-1]}},\\
\mathfrak I^{(k,\gamma^{(k)})}(t,\spadesuit^{(k)})
&=\int_0^t e^{-\iu Q(\lambda(\spadesuit^{(k)}))(t-s)}
\prod_{j=1}^3
\left\{\mathfrak I^{(k-1,\gamma_j^{(k-1)})} (s,\spadesuit_j^{(k-1)})\right\}^{\ast^{[j-1]}}
\,\mathrm{d}s,\\
\mathfrak F^{(k,\gamma^{(k)})}(\spadesuit^{(k)})
&=\mathrm{i}{\varepsilon^\alpha} \prod_{j=1}^3
\left\{\mathfrak F^{(k-1,\gamma_j^{(k-1)})}(\spadesuit_j^{(k-1)})\right\}^{\ast^{[j-1]}}.
\end{align*}

\end{proposition}

\begin{proof}
The result follows by iterating \eqref{eq:picard} with initial condition \eqref{eq:initialguess}
as in \cite[Theorem 2.1]{X25}.
\end{proof}
\subsection{Uniform-in-Time Polynomial Decay of the Picard Sequence}
Let $\boldsymbol{n}\in\mathbb{Z}^\nu$ be fixed. We also expand it as
\[
(n_1,\cdots,n_d)\in\mathbb{Z}^{\nu_1}\times\cdots\times\mathbb{Z}^{\nu_d},
\]
and further write
\[(n_{1,1},\cdots,n_{1,\nu_1};\cdots;n_{d,1},\cdots,n_{d,\nu_d})\in\underbrace{\mathbb Z\times\cdots\times\mathbb Z}_{\nu_1}\times\cdots\times\underbrace{\mathbb Z\times\cdots\times\mathbb Z}_{\nu_d}.\]

For $\gamma^{(k)}=0\in\Gamma^{(k)}$ with $k\geq1$ (see \autoref{fig:diagramlinear}), we set $\boldsymbol{m}^1=\boldsymbol{n}$, that is,
$(m^1_1,\cdots,m^1_d)=(n_1,\cdots,n_d)$, or equivalently,
\[
\left(m^1_{1,1},\cdots,m^1_{1,\nu_1}; \cdots; m^1_{d,1},\cdots, m^1_{d,\nu_d}\right)
=
\left(n_{1,1},\cdots,n_{1,\nu_1}; \cdots; n_{d,1},\cdots, n_{d,\nu_d}\right).
\]

For $\gamma^{(k)}=(\gamma_1^{(k-1)},\gamma_2^{(k-1)},\gamma_3^{(k-1)})\in\Gamma^{(k)}\setminus\{0\}$, we write $\boldsymbol{n}=\boldsymbol{n}^1-\boldsymbol{n}^2+\boldsymbol{n}^3\in\mathbb Z^\nu$ (see \autoref{fig:diagramduhamel}), and set
\begin{align*}
\mathbb Z^\nu\ni\boldsymbol{m}^j=
\begin{cases}
\boldsymbol{n}^{1,j}, & 1\leq j\leq 2\sigma(\gamma_1^{(k-1)});\\[4pt]
\boldsymbol{n}^{2,\,j-2\sigma(\gamma_1^{(k-1)})}, 
& 2\sigma(\gamma_1^{(k-1)})+1\le j\le 2\sigma(\gamma_1^{(k-1)})+2\sigma(\gamma_2^{(k-1)});\\[4pt]
\boldsymbol{n}^{3,\,j-2\sigma(\gamma_1^{(k-1)})-2\sigma(\gamma_2^{(k-1)})}, 
& 2\sigma(\gamma_1^{(k-1)})+2\sigma(\gamma_2^{(k-1)})+1\le j\le 2\sigma(\gamma^{(k)}).
\end{cases}
\end{align*}
Equivalently, for each component $s=1,\dots,d$, we have
\begin{align*}
\mathbb Z^{\nu_{s}}\ni m_{s}^j=
\begin{cases}
n_{s}^{1,j}, & 1\leq j\leq 2\sigma(\gamma_1^{(k-1)});\\[4pt]
n_{s}^{2,\,j-2\sigma(\gamma_1^{(k-1)})}, 
& 2\sigma(\gamma_1^{(k-1)})+1\le j\le 2\sigma(\gamma_1^{(k-1)})+2\sigma(\gamma_2^{(k-1)});\\[4pt]
n_{s}^{3,\,j-2\sigma(\gamma_1^{(k-1)})-2\sigma(\gamma_2^{(k-1)})}, 
& 2\sigma(\gamma_1^{(k-1)})+2\sigma(\gamma_2^{(k-1)})+1\le j\le 2\sigma(\gamma^{(k)}),
\end{cases}
\end{align*}
where $m_{s}^j=\left(m^j_{s,1},\cdots,m^j_{s,\nu_{s}}\right)\in\mathbb Z^{\nu_{s}}$. 

Thus
\[
\widetilde{\spadesuit}^{(k)}_{x_s}=(m_s^j)_{1\le j\le2\sigma(\gamma^{(k)})},\qquad 1\le s\le d,\quad k\geq1. 
\]

\begin{proposition}[Estimates of $\mathfrak C,\mathfrak I$ and $\mathfrak F$]\label{lemma:cif}
For all $k\geq1$, the following estimates hold: 

\begin{enumerate}[label=(\arabic*)]

\item 
\begin{align}\label{ce}
\mathfrak C^{(k,\gamma^{(k)})}(\spadesuit^{(k)})
=\prod_{j=1}^{2\sigma(\gamma^{(k)})}\left\{c(\boldsymbol m^j)g_{\boldsymbol{m}^j}\right\}^{\ast^{[j-1]}}.
\end{align}
Moreover, under the decay property \eqref{eq:decayall}, it holds that
\begin{align}\label{ece}
|\mathfrak C^{(k,\gamma^{(k)})}(\spadesuit^{(k)})|&\leq
{\varepsilon^{-(1+\eta)\sigma(\gamma^{(k)})}}\prod_{j=1}^{2\sigma(\gamma^{(k)})}\tnorm{\boldsymbol{m}^j}_{-(\rho-\kappa)}\nonumber\\
&={\varepsilon^{-(1+\eta)\sigma(\gamma^{(k)})}}
\prod_{j=1}^{2\sigma(\gamma^{(k)})}\prod_{s=1}^{d}\prod_{q=1}^{\nu_s}\left(1+|m^{j}_{s,q}|\right)^{-(\rho_{s,q}-\kappa_{s,q})}; 
\end{align}

\item 
\begin{align}\label{ie}
|\mathfrak I^{(k,\gamma^{(k)})}(t,\spadesuit^{(k)})|\leq\frac{t^{\ell(\gamma^{(k)})}}{\mathfrak D(\gamma^{(k)})},
\end{align}
where
\begin{align*}
\mathfrak D(\gamma^{(k)}) =
\begin{cases}
1, & \gamma^{(k)}=0\in\Gamma^{(k)},\ k\geq1,\\[4pt]
1, & \gamma^{(1)}\in\Gamma^{(1)},\\[4pt]
\displaystyle \ell(\gamma^{(k)})\prod_{j=1}^3\mathfrak D(\gamma_j^{(k-1)}), & \gamma^{(k)}=(\gamma_j^{(k-1)})\in(\Gamma^{(k-1)})^3; 
\end{cases}
\end{align*}

\item 
\begin{align}\label{fe}
|\mathfrak F^{(k,\gamma^{(k)})}(\spadesuit^{(k)})|\leq {\varepsilon^{\alpha\ell(\gamma^{(k)})}}.
\end{align}

\end{enumerate}
\end{proposition}
\begin{proof}
These estimates are proved by induction, following the argument in \cite[Lemmas 2.5--2.7]{X25}.
\end{proof}

\begin{lemma}[Uniform-in-time polynomial decay of ${c}_k(t,\boldsymbol{n})$]\label{size1}
{On the complement of $\Omega_\delta$ with $\delta\sim\varepsilon^{1+\eta}$}, for all $k\geq1$, it holds that
\begin{align}\label{eq:decayt}
\sup_{t\in[0,T_\varepsilon]}|{c}_k(t,\boldsymbol{n})|\lesssim 
{\varepsilon^{-\frac{1}{2}-\frac{\eta}{2}}}\tnorm{\boldsymbol{n}}_{-\frac{\rho-\kappa}{2}}\end{align}
\end{lemma}

\begin{proof}
It follows from \autoref{prop:tree} and \autoref{lemma:cif} that

\begin{align*}
|c_k(t,\boldsymbol{n})|\leq
{\varepsilon^{-\frac{1}{2}-\frac{\eta}{2}}}
\sum_{\gamma^{(k)}\in\Gamma^{(k)}}\frac{({\varepsilon^{\alpha-1-\eta}} t)^{\ell(\gamma^{(k)})}}{\mathfrak D(\gamma^{(k)})}
\sum_{\substack{\lambda(\spadesuit^{(k)})=\boldsymbol{n}\\\spadesuit^{(k)}\in\mathrm D^{(k,\gamma^{(k)})}}}\prod_{j=1}^{2\sigma(\gamma^{(k)})}
\prod_{s=1}^{d}\prod_{q=1}^{\nu_s}\left(1+|m^{j}_{s,q}|\right)^{-(\rho_{s,q}-\kappa_{s,q})}.
\end{align*}

We split the weight into two factors, each carrying exponent $-\frac{\rho_{s,q}-\kappa_{s,q}}{2}$. 

On the one hand, we estimate one factor using the Bernoulli's inequality and obtain the following result:  
\begin{align*}
\prod_{j=1}^{2\sigma(\gamma^{(k)})}
\prod_{s=1}^{d}\prod_{q=1}^{\nu_s}\left(1+|m^{j}_{s,q}|\right)^{-(\rho_{s,q}-\kappa_{s,q})}
&=\prod_{s=1}^{d}\prod_{q=1}^{\nu_s}
\left\{\prod_{j=1}^{2\sigma(\gamma^{(k)})}
\left(1+|m^{j}_{s,q}|\right)\right\}^{-\frac{\rho_{s,q}-\kappa_{s,q}}{2}}\\
&\leq\prod_{s=1}^{d}\prod_{q=1}^{\nu_s}
\left\{1+\sum_{j=1}^{2\sigma(\gamma^{(k)})}
|m^{j}_{s,q}|\right\}^{-\frac{\rho_{s,q}-\kappa_{s,q}}{2}}\\
&\leq\prod_{s=1}^{d}\prod_{q=1}^{\nu_s}
\left\{1+\left|\sum_{j=1}^{2\sigma(\gamma^{(k)})}
(-1)^{j-1}m^{j}_{s,q}\right|\right\}^{-\frac{\rho_{s,q}-\kappa_{s,q}}{2}}\\
&=\prod_{s=1}^{d}\prod_{q=1}^{\nu_s}
\left(1+|n_{s,q}|\right)^{-\frac{\rho_{s,q}-\kappa_{s,q}}{2}}\\
&=\tnorm{\boldsymbol{n}}_{-\frac{\rho-\kappa}{2}}.
\end{align*}
Here we use the condition $\lambda(\spadesuit^{(k)})=\boldsymbol{n}$, that is,
\[\sum_{j=1}^{2\sigma(\gamma^{(k)})}(-1)^{j-1}m^j_{s,q}=n_{s,q},\qquad 1\leq s\leq d\quad\text{and}\quad1\leq q\leq\nu_s.\]

On the other hand, we have
\begin{align*}
\sum_{\substack{\lambda(\spadesuit^{(k)})=\boldsymbol{n}\\\spadesuit^{(k)}\in\mathrm D^{(k,\gamma^{(k)})}}}
\prod_{j=1}^{2\sigma(\gamma^{(k)})}
\prod_{s=1}^{d}\prod_{q=1}^{\nu_s}\left(1+|m^{j}_{s,q}|\right)^{-\frac{\rho_{s,q}-\kappa_{s,q}}{2}}
&\leq \prod_{j=1}^{2\sigma(\gamma^{(k)})}
\prod_{s=1}^{d}\prod_{q=1}^{\nu_s}\sum_{m_{s,q}^{j}\in\mathbb Z}\left(1+|m^{j}_{s,q}|\right)^{-\frac{\rho_{s,q}-\kappa_{s,q}}{2}}\\
&\leq \prod_{j=1}^{2\sigma(\gamma^{(k)})}
\prod_{s=1}^{d}\prod_{q=1}^{\nu_s}\sum_{m_{s,q}^{j}\in\mathbb Z}\left(1+|m^{j}_{s,q}|\right)^{-\frac{(\rho-\kappa)_{\min}}{2}}\\
&\leq \prod_{j=1}^{2\sigma(\gamma^{(k)})}
\prod_{s=1}^{d}\prod_{q=1}^{\nu_s}
\left\{1+2\mathfrak{b}\left(\frac{(\rho-\kappa)_{\min}}{2}\right)\right\}\\
 &\leq \mathfrak{B}^{2\sigma(\gamma^{(k)})\nu}\\
&= \mathfrak{B}^{(2\ell(\gamma^{(k)})+1)\nu}\\
&= \mathfrak{B}^{\nu} \cdot \left(\mathfrak{B}^{2\nu}\right)^{\ell(\gamma^{(k)})},
\end{align*}
where
\begin{align*}
(\rho-\kappa)_{\text{min}}&=\min_{\substack{1\leq j\leq d\\1\le j^\prime\le\nu_j}}\rho_{j,j^\prime}-\kappa_{j,j^\prime};\\
\mathfrak{b}(s) &= \frac{s}{s-1} \ge \zeta(s) = \sum_{n=1}^{\infty} \frac{1}{n^s}, \quad s > 1;\\
\mathfrak{B} &= 1 + 2\mathfrak{b}\left(\frac{(\rho-\kappa)_{\min}}{2}\right).
\end{align*}

Thus
\[
\big|c_k(t,\boldsymbol{n})\big|
\lesssim {\varepsilon^{-\frac{1}{2}-\frac{\eta}{2}}}
\tnorm{\boldsymbol{n}}_{-\frac{\rho-\kappa}{2}}
\sum_{\gamma^{(k)}\in\Gamma^{(k)}}
\frac{({{\varepsilon^{\alpha-1-\eta}}\mathfrak B^{2\nu}} t)^{\ell(\gamma^{(k)})}}{\mathfrak D(\gamma^{(k)})}.
\]

It follows from \cite[Lemma 4.4]{DLX24} that the above series over $\Gamma^{(k)}$ is uniformly bounded provided that
\[
0<t\leq\frac{4}{27}\mathfrak B^{-2\nu}{\varepsilon^{-\alpha+1+\eta}}\triangleq T_{\varepsilon}\lesssim{\varepsilon^{-\alpha+1+\eta}}.
\]

Consequently, \eqref{eq:decayt} holds. 
\end{proof}

\section{Nonlinear LDP}\label{sec:nonlinear}
In this section, we combine the linear LDP (\autoref{thm:expectation-variance}) with the well-posedness result (\autoref{thm:eude}) to establish the nonlinear LDP (\autoref{thm:nonlinearldp}). The main idea is to transfer the large deviations estimates from the linearized equation to the full nonlinear dynamics via a suitable {\em perturbative} argument. In particular, the well-posedness theory guarantees that the nonlinear solution can be constructed globally in the regime under consideration, which allows us to compare it with the corresponding linear evolution and control the nonlinear error terms; see \autoref{eq:duhamel}.

\subsection{Duhamel Estimates}

\begin{lemma}[Duhamel estimates]\label{eq:duhamel}
Let $u$ be the nonlinear solution to the Cauchy problem \eqref{eq:nls}--\eqref{eq:initialdata}, and let $u_{\mathrm{linear}}$ denote the corresponding linearized solution (see \eqref{eq:linearsolution}). {On the complement of $\Omega_\delta$ with $\delta\sim\varepsilon^{1+\eta}$}, we have
\[
\|u(t)-u_{\mathrm{linear}}(t)\|_{L_{\boldsymbol{x}}^\infty(\mathbb{R}^d)}
=\mathcal O\left(\varepsilon^{-\frac{1}{2}+\frac{\mu_3}{2}}\right),
\quad 
t\sim\varepsilon^{-\alpha+1+\frac{3}{2}\eta+\frac{\mu_3}{2}}.
\]
\end{lemma}

\begin{proof}
We begin by writing
\begin{align*}
\| u(t)- u_{\mathrm{linear}}(t)\|_{L_x^\infty(\mathbb R^d)}
&\leq \sum_{\boldsymbol n\in\mathbb Z^\nu} |(c - c_0)(t,\boldsymbol n)|.
\end{align*}
Using the Duhamel expansion for the nonlinear evolution, we obtain
\begin{align*}
\| u(t)- u_{\mathrm{linear}}(t)\|_{L_x^\infty(\mathbb R^d)}
&\leq {\varepsilon^{\alpha}} \int_0^t \sum_{\boldsymbol n\in\mathbb Z^\nu}
\sum_{\substack{\boldsymbol n^j\in\mathbb Z^\nu,\; j=1,2,3\\
\boldsymbol n^1-\boldsymbol n^2+\boldsymbol n^3=\boldsymbol n}}
\prod_{j=1}^3 |c(s,\boldsymbol n^j)|\, \mathrm{d}s.
\end{align*}
We then apply the decay structure \eqref{eq:decayt} of the coefficients and obtain
\begin{align*}
\| u(t)- u_{\mathrm{linear}}(t)\|_{L_x^\infty(\mathbb R^d)}
&\lesssim {\varepsilon^{\alpha-\frac{3}{2}-\frac{3}{2}\eta}}\, \varepsilon^{-\alpha+1+\frac{3}{2}\eta+\frac{\mu_3}{2}}
\prod_{j=1}^3 \sum_{\boldsymbol n^j\in\mathbb Z^\nu}
\tnorm{\boldsymbol{n}^j}_{-\frac{\rho-\kappa}{2}}\\
&\lesssim {\varepsilon^{-\frac{1}{2}+\frac{\mu_3}{2}}}\,
\prod_{j=1}^3 \sum_{\boldsymbol n^j\in\mathbb Z^\nu}
\tnorm{\boldsymbol{n}^j}_{-\frac{\rho-\kappa}{2}}.
\end{align*}
Since the weighted sums are finite under the condition
\[
\rho_{j,j'} - \kappa_{j,j'}>2, \quad j^\prime=1,\cdots,\nu_j, j=1,\cdots,d,
\]
we conclude that
\[
\| u(t)- u_{\mathrm{linear}}(t)\|_{L_x^\infty(\mathbb R^d)}
\lesssim \varepsilon^{-\frac{1}{2}+\frac{\mu_3}{2}}.
\]
\end{proof}
\begin{remark}
By \autoref{eq:duhamel}, there exists a positive constant $C$ such that
\[
\| u(t)- u_{\mathrm{linear}}(t)\|_{L_x^\infty(\mathbb R^d)}
\le C\varepsilon^{-\frac{1}{2}+\frac{\mu_3}{2}}. 
\]
Set
\[
\widehat{\Omega}=\left\{\| u(t)- u_{\mathrm{linear}}(t)\|_{L_x^\infty(\mathbb R^d)}
\leq C\varepsilon^{-\frac{1}{2}+\frac{\mu_3}{2}}\right\}.
\]
Clearly, 
\begin{align}\label{eq:hatomega}
\Omega_\delta^c\subset\widehat{\Omega}.
\end{align}
\end{remark}

\subsection{Lower bound}
Set
\begin{align*}
\mathscr A_\varepsilon^- &= \left\{ \| u_{\text{linear}} \|_{L^\infty_x(\mathbb{R}^d)} - \| u - u_{\text{linear}} \|_{L^\infty_x(\mathbb{R}^d)} > z_0 \, \varepsilon^{-\frac{1}{2}} \right\},\\
\mathscr B_\varepsilon^- &= \left\{ \| u_{\text{linear}} \|_{L^\infty_x(\mathbb{R}^d)} > \left( z_0 + C \, \varepsilon^{\frac{\mu_3}{2}} \right) \varepsilon^{-\frac{1}{2}} \right\}.
\end{align*}
Clearly, we have the following inclusion relation: 
\[
\mathscr B_\varepsilon^-\cap\widehat{\Omega} 
\subset \mathscr A_\varepsilon^-
\subset \mathscr A_\varepsilon^\infty.
\]

By the law of total probability and \eqref{eq:hatomega}, we have
\begin{align*}
\mathbb P(\mathscr A_\varepsilon^\infty)
&\ge \mathbb P(\mathscr B_\varepsilon^-\cap\widehat{\Omega})\\
&= \mathbb P(\mathscr B_\varepsilon^-)
   -\mathbb P(\mathscr B_\varepsilon^-\cap\widehat{\Omega}^{\,c})\\
&\ge \mathbb P(\mathscr B_\varepsilon^-)-\mathbb P(\widehat{\Omega}^{\,c})\\
&\ge \mathbb P(\mathscr B_\varepsilon^-)-\mathbb P(\Omega_\delta)\\
&=\mathbb P(\mathscr B_\varepsilon^-)\left(1-\frac{\mathbb P(\Omega_\delta)}{\mathbb P(\mathscr B_\varepsilon^-)}\right).
\end{align*}
Taking logarithms and multiplying by \(\varepsilon\), we have
\begin{align}\label{w2}
\varepsilon\log\mathbb P(\mathscr A_\varepsilon^\infty)\geq
\varepsilon\log\mathbb P(\mathscr B_\varepsilon^-)
+\varepsilon\log\left(1-\frac{\mathbb P(\Omega_\delta)}{\mathbb P(\mathscr B_\varepsilon^-)}\right).
\end{align}

Applying \autoref{thm:linearldp}, we deduce that
\begin{align*}
\lim_{\varepsilon \to 0^+} \varepsilon \log\mathbb P \left(\mathscr B_\varepsilon^- \right) &= -\frac{z_0^2}{\sum_{\boldsymbol{n}\in\mathbb Z^\nu}|  c(\boldsymbol n)|^2}. 
\end{align*}

Similar to the analysis in \autoref{sec:all}, 
\[\lim_{\varepsilon\to0}\varepsilon\log\left(1-\frac{\mathbb P(\Omega_\delta)}{\mathbb P(\mathscr B_\varepsilon^-)}\right)=0.\]

Therefore
\begin{align}\label{eqlower}
\liminf_{\varepsilon\to0^+}\varepsilon\log\mathbb P(\mathscr A_\varepsilon^\infty)
\ge-\frac{z_0^2}{\sum_{\boldsymbol{n}\in\mathbb Z^\nu}|  c(\boldsymbol n)|^2}.
\end{align}
\subsection{Upper bound}
Set
\begin{align*}
\mathscr A_\varepsilon^+ &= \left\{ \| u_{\text{linear}} \|_{L^\infty_x(\mathbb{R}^d)} + \| u - u_{\text{linear}} \|_{L^\infty_x(\mathbb{R}^d)} > z_0 \, \varepsilon^{-\frac{1}{2}} \right\},\\
\mathscr B_\varepsilon^+ &= \left\{ \| u_{\text{linear}} \|_{L^\infty_x(\mathbb{R}^d)} > \left( z_0 - C \, \varepsilon^{\frac{\mu_3}{2}} \right) \varepsilon^{-\frac{1}{2}} \right\}.
\end{align*}
Clearly, we have the following inclusion relation: 
\[
\mathscr A_\varepsilon^\infty \subset \mathscr A_\varepsilon^+,\quad \mathscr A_\varepsilon^+\cap\widehat{\Omega} \subset \mathscr B_\varepsilon^+.
\]

It follows from the law of total probability and \eqref{eq:hatomega} that
\begin{align*}
\mathbb P(\mathscr A_\varepsilon^\infty)
&\le \mathbb P(\mathscr A_\varepsilon^+)\\
&=\mathbb P(\mathscr A_\varepsilon^+\cap\widehat{\Omega})
+\mathbb P(\mathscr A_\varepsilon^+\cap\widehat{\Omega}^c)\\
&\le \mathbb P(\mathscr B_\varepsilon^+)+\mathbb P(\widehat{\Omega}^c)\\
&\le \mathbb P(\mathscr B_\varepsilon^+)+\mathbb P(\Omega_\delta)\\
&=\mathbb P(\mathscr B_\varepsilon^+)\left(1+\frac{\mathbb P(\Omega_\delta)}{\mathbb P(\mathscr B_\varepsilon^+)}\right).
\end{align*}
Taking logarithms and multiplying by \(\varepsilon\), we have
\begin{align}\label{w3}
\varepsilon\log\mathbb P(\mathscr A_\varepsilon^\infty)\le
\varepsilon\log\mathbb P(\mathscr B_\varepsilon^+)
+\varepsilon\log\left(1+\frac{\mathbb P(\Omega_\delta)}{\mathbb P(\mathscr B_\varepsilon^+)}\right).
\end{align}

Applying \autoref{thm:linearldp}, we deduce that
\begin{align*}
\lim_{\varepsilon \to 0^+} \varepsilon \log \mathbb{P}\left(\mathscr B_\varepsilon^+ \right) &= -\frac{z_0^2}{\sum_{\boldsymbol{n}\in\mathbb Z^\nu}|  c(\boldsymbol n)|^2}. 
\end{align*}

Analogously, 
\[\lim_{\varepsilon\to0}\varepsilon\log\left(1+\frac{\mathbb P(\Omega_\delta)}{\mathbb P(\mathscr B_\varepsilon^+)}\right)=0.\]

Hence we obtain
\begin{align}\label{equpper}
\limsup_{\varepsilon\to0^+}\varepsilon\log\mathbb P(\mathscr A_\varepsilon^\infty)
\le-\frac{z_0^2}{\sum_{\boldsymbol{n}\in\mathbb Z^\nu}|  c(\boldsymbol n)|^2}.
\end{align}

Combining \eqref{eqlower} and \eqref{equpper} yields that
\begin{align*}
\lim_{\varepsilon\to0^+}\varepsilon\log\mathbb P\Big( \|u(t)\|_{L_{\boldsymbol x}^\infty(\mathbb R^d)} > z_0 \varepsilon^{-\frac{1}{2}} \Big)
=-\frac{z_0^2}{\sum_{\boldsymbol n\in\mathbb Z^\nu}|c(\boldsymbol n)|^2},\quad t\sim\varepsilon^{-\alpha+1+\frac{3}{2}\eta+\frac{\mu_3}{2}}.
\end{align*}
This completes the proof of \autoref{thm:nonlinearldp}. 

\appendix

\section{Probability and Statistics}\label{appendixps}
In this section, we review some fundamental notation and concepts from probability and statistics.

\subsection{Real-Valued Random Variables}
Let the triple $(\Omega, \mathcal{F}, \mathbb{P})$ be a \emph{probability space}. 
The set $\Omega$ denotes the \emph{sample space}, comprising all possible outcomes, 
and $\omega \in \Omega$ represents an individual outcome (\emph{sample point}). 
The collection $\mathcal{F}$ is a $\sigma$-algebra on $\Omega$, referred to as the \emph{event space}, 
whose elements are subsets of $\Omega$; these events are typically denoted by uppercase letters such as $A$, $B$, $C$. 
The function $\mathbb{P} : \mathcal{F} \to [0,1]$ is a \emph{probability measure}, 
assigning to each event $A \in \mathcal{F}$ its probability $\mathbb{P}(A)$.

Let $\mathbb{K} \in \{\mathbb{R}, \mathbb{C}\}$ be a field. 
A {\em $\mathbb{K}$-valued random variable} $X$ is a measurable function 
from the sample space to $\mathbb{K}$, that is, 
$X : \Omega \to \mathbb{K}$, $\omega \mapsto X(\omega)$.

We now focus on real-valued random variables, i.e., the case $\mathbb{K} = \mathbb{R}$. 
Events of the form $\{\omega \in \Omega : X(\omega) \in D \subset \mathbb{R}\} := A$ 
are commonly abbreviated as $\{X \in D\}$. 
Furthermore, $\mathbb{P}(A) = \mathbb{P}(\{X \in D\}) \triangleq \mathbb{P}(X \in D)$.

The {\em cumulative distribution function} (CDF) of a random variable $X$, 
denoted by $F_X$, is defined as $F_X(x) = \mathbb{P}(X \leq x)$ for all $x \in \mathbb{R}$.

Under standard regularity conditions, such as the differentiability of $F_X$ 
with respect to $x$, the {\em probability density function} (PDF) of $X$, 
denoted $f_X$, is defined as the derivative
$$
f_X(x) = \frac{{\rm d}}{{\rm d}x} F_X(x).
$$

The {\em expectation}, or {\em mean}, of a random variable $X$, denoted $\mathbb{E}[X]$, 
is the weighted average of $X(\omega)$ (or $X$) using the probability measure 
${\rm d}\mathbb{P}(\omega)$ (or PDF $f_X$) as weights, that is, 
\begin{align*}
\mathbb{E}[X] &= \int_\Omega X(\omega)\,{\rm d}\mathbb{P}(\omega) \quad (\text{using the probability measure in abstract space}), \\
&= \int_{\mathbb{R}} x\,{\rm d}F_X(x) \quad (\text{using the CDF via Stieltjes integral}), \\
&= \int_{\mathbb{R}} x f_X(x)\,{\rm d}x \quad (\text{using the PDF in the absolutely continuous case}).
\end{align*}
We also denote this by $\mu_X \triangleq \mu$.

The expectation satisfies linearity, that is, 
$$
\mathbb{E}[a_1X_1 + a_2X_2] = a_1\mathbb{E}[X_1] + a_2\mathbb{E}[X_2], \quad \forall a_1,a_2 \in \mathbb{R}.
$$

The {\em variance} of a random variable $X$, denoted $\mathbb{V}(X)$, is defined as 
the expected value of $(X-\mu_X)^2$, that is,
$$
\mathbb{V}(X) = \mathbb{E}[(X-\mu_X)^2].
$$
We also denote this by $\sigma_X^2 \triangleq \sigma^2$. By linearity of expectation,
$$
\mathbb{V}(X) = \mathbb{E}[X^2] - \mu_X^2.
$$
Furthermore,
$$
\mathbb{V}(aX + b) = a^2 \mathbb{V}(X), \quad \forall a \in \mathbb{R}.
$$

The {\em characteristic function} (CF) of $X$, denoted $\phi_X$, is defined by
$$
\phi_X(t) = \mathbb{E}[e^{{\rm i} t X}], \quad t \in \mathbb{R}.
$$
It should be noted that characteristic functions uniquely determine the distribution: 
$\phi_X = \phi_Y$ if and only if $X$ and $Y$ have the same distribution. 
This follows from the L{\'e}vy continuity theorem; see \cite{Tao12}.

\subsection{Several Important Distributions}
In this subsection, we briefly introduce several important classical distributions 
relevant to studying rogue wave mechanisms from a probabilistic perspective. 
These distributions are repeatedly applied in the calculation of LDP for rogue waves.

\subsubsection{Gaussian Distribution}
We say that a random variable $X$ follows the {\em Gaussian/normal distribution} 
with mean $\mu$ and variance $\sigma^2$ if its PDF is
$$
f_X(x;\mu,\sigma) = \frac{1}{\sqrt{2\pi}\sigma} e^{-\frac{(x-\mu)^2}{2\sigma^2}}, \quad x \in \mathbb{R}.
$$
The corresponding mean, variance and characteristic function are, respectively, 
$$
\mathbb{E}[X] = \mu, \quad \mathbb{V}(X) = \sigma^2, \quad \phi_X(t) = e^{{\rm i}\mu t - \frac{1}{2}\sigma^2 t^2}.
$$

Here and throughout, we write $X \sim \mathscr{N}_\mathbb{R}(\mu,\sigma^2)$ to denote 
a real-valued random variable following this distribution.

If $X \sim \mathscr{N}_\mathbb{R}(\mu,\sigma^2)$, then the normalization 
$Y = \frac{X-\mu}{\sigma} \sim \mathscr{N}_\mathbb{R}(0,1)$. 
In this case, we say that $Y$ follows the {\em standard} Gaussian/normal distribution, 
with PDF $f_Y(y) = \frac{1}{\sqrt{2\pi}}e^{-\frac{1}{2}y^2}$ and CF 
$\phi_Y(t) = e^{-\frac{1}{2}t^2}$.

\subsubsection{Exponential Distribution}\label{sec:exponentialdistribution}

We say that a random variable $X$ follows the {\em exponential distribution} 
with {\em rate parameter} $\lambda > 0$ if its PDF is
$$
f_X(x;\lambda) = \lambda e^{-\lambda x} \boldsymbol{1}_{[0,\infty)}(x).
$$

Throughout, we write $X \sim \mathscr{E}(\lambda)$ to denote that $X$ follows 
this distribution.

If $X \sim \mathscr{E}(\lambda)$, then
$$
\mathbb{E}[X] = \frac{1}{\lambda}, \quad \mathbb{V}(X) = \frac{1}{\lambda^2},
$$
and $aX \sim \mathscr{E}\bigl(\frac{\lambda}{a}\bigr)$ for $a > 0$.

\subsubsection{Chi-squared Distribution}

We say that a random variable $X$ follows the {\em chi-squared distribution} 
with $n$ degrees of freedom if its PDF is
$$
f_X(x;n) = \frac{1}{\Gamma\left(\frac{n}{2}\right)2^{n/2}} x^{\frac{n}{2}-1} e^{-x/2} \boldsymbol{1}_{[0,\infty)}(x).
$$

Throughout, we write $X \sim \chi_n^2$ to denote this distribution.

If $X \sim \chi_n^2$, then
$$
\mathbb{E}[X] = n, \quad \mathbb{V}(X) = 2n.
$$

In addition, we have
\begin{itemize}
\item $\chi_2^2 = \mathscr{E}\bigl(\frac{1}{2}\bigr)$.
\item If $X_1 \sim \chi_{n_1}^2$ and $X_2 \sim \chi_{n_2}^2$ are independent, then $X_1 + X_2 \sim \chi_{n_1+n_2}^2$.
\item If $\{X_j\}_{1\leq j\leq n}$ are i.i.d. $\mathscr{E}(\lambda)$, then $2\lambda\sum_{j=1}^n X_j \sim \chi_{2n}^2$.
\item If $\{X_j\}_{1\leq j\leq n}$ are i.i.d. $\mathscr{N}_\mathbb{R}(0,1)$, then $\sum_{j=1}^n X_j^2 \sim \chi_n^2$.
\end{itemize}

To introduce complex-valued random variables, we first review the concepts of 
joint distributions (see \autoref{sec:jointdistribution}) and random independence 
(see \autoref{sec:randomindependence}), initially for two random variables. 
These concepts naturally generalize to any finite collection of random variables.

\subsection{Joint Distribution}\label{sec:jointdistribution}

Let $X$ and $Y$ be random variables defined on a common probability space 
$(\Omega, \mathcal{F}, \mathbb{P})$. The {\em joint CDF} of $(X,Y)$, denoted $F_{X,Y}$, 
is defined by
$$
F_{X,Y}(x,y) = \mathbb{P}(X \leq x, Y \leq y), \quad (x,y) \in \mathbb{R}^2.
$$

Under standard regularity conditions (such as the existence of mixed partial derivatives), 
the {\em joint PDF}, denoted $f_{X,Y}$, is given by
$$
f_{X,Y}(x,y) = \frac{\partial^2 F_{X,Y}(x,y)}{\partial x \partial y}.
$$

Let $f_X$ and $f_Y$ denote the marginal PDFs of $X$ and $Y$. Then
\begin{align*}
f_X(x) &= \int_{\mathbb{R}} f_{X,Y}(x,y)\,{\rm d}y = \frac{\partial F_{X,Y}(x,y)}{\partial x}\Big|_{y\to +\infty}, \\
f_Y(y) &= \int_{\mathbb{R}} f_{X,Y}(x,y)\,{\rm d}x = \frac{\partial F_{X,Y}(x,y)}{\partial y}\Big|_{x\to +\infty}.
\end{align*}
Here, $f_X$ and $f_Y$ are called the {\em marginal distributions}.

The {\em joint CF} for the joint distribution of $(X,Y)$, 
denoted by $\phi_{X,Y}$, is defined by
$$
\phi_{X,Y}(t_1,t_2) = \mathbb{E}[e^{{\rm i}(t_1X + t_2Y)}], \quad (t_1,t_2) \in \mathbb{R}^2.
$$

\subsection{Conditional Probability and Independence}\label{sec:randomindependence}

Let $A$ and $B$ be events in the probability space $(\Omega, \mathcal{F}, \mathbb{P})$ 
with $\mathbb{P}(B) > 0$. The {\em conditional probability} of $A$ given $B$, denoted 
$\mathbb{P}(A|B)$, is defined by
$$
\mathbb{P}(A|B) = \frac{\mathbb{P}(A \cap B)}{\mathbb{P}(B)}.
$$
This yields the {\em multiplication rule}:
$$
\mathbb{P}(A \cap B) = \mathbb{P}(A|B) \mathbb{P}(B).
$$

Events $A$ and $B$ are {\em independent} if
$$
\mathbb{P}(A|B) = \mathbb{P}(A),
$$
or equivalently,
$$
\mathbb{P}(A \cap B) = \mathbb{P}(A) \mathbb{P}(B).
$$
In this case, $\mathbb{P}(B|A) = \mathbb{P}(B)$ also holds.

Let $X$ and $Y$ be random variables defined on a common probability space 
$(\Omega, \mathcal{F}, \mathbb{P})$. We say $X$ and $Y$ are {\em independent} 
if the events $\{X \leq x\}$ and $\{Y \leq y\}$ are independent for all 
$x,y \in \mathbb{R}$. Specifically,
$$
\mathbb{P}(X \leq x, Y \leq y) = \mathbb{P}(X \leq x) \mathbb{P}(Y \leq y).
$$
That is,
\begin{align*}
F_{X,Y}(x,y) &= F_X(x) F_Y(y) \quad (\text{using the CDF}), \\
f_{X,Y}(x,y) &= f_X(x) f_Y(y) \quad (\text{using the PDF}), \\
\phi_{X,Y}(t_1,t_2) &= \phi_X(t_1) \phi_Y(t_2) \quad (\text{using the CF}).
\end{align*}

\subsection{Complex-Valued Random Variables}

From now on, we focus on {\em complex-valued random variables}, i.e., $\mathbb{K} = \mathbb{C}$, 
particularly those following {\em complex Gaussian distributions}.

Let $Z$ be a complex random variable on the probability space $(\Omega, \mathcal{F}, \mathbb{P})$, 
that is, $Z : \Omega \to \mathbb{C}$, $\omega \mapsto Z(\omega)$.

The distribution of $Z$ admits two natural representations. 
The first, based on the {\em Cartesian coordinate system}, identifies it with the 
joint distribution of the real and imaginary parts $\operatorname{Re}Z$ and $\operatorname{Im}Z$. 
The second, based on the {\em polar coordinate system}, identifies it with the 
joint distribution of the modulus $|Z|$ and argument $\arg Z$.

In the {\em Cartesian approach}, write $Z(\omega) = X(\omega) + {\rm i} Y(\omega)$ 
for $\omega \in \Omega$, where $X,Y$ are real-valued random variables on 
$(\Omega, \mathcal{F}, \mathbb{P})$. The distribution of $Z$ is the joint distribution 
of $(X,Y)$ under the {\em partial order} $\prec$ defined by
$$
z_1 \prec z_2 \Longleftrightarrow \operatorname{Re} z_1 \leq \operatorname{Re} z_2 
\;\text{and}\; \operatorname{Im} z_1 \leq \operatorname{Im} z_2.
$$
For $z = x + {\rm i} y \in \mathbb{C}$, the {\em CDF} is
$$
F_Z(z) = \mathbb{P}(Z \prec z) = \mathbb{P}(X \leq x, Y \leq y) = F_{X,Y}(x,y).
$$
Hence the {\em PDF} is given by 
\[f_Z(z)=f_{X,Y}(x, y)=\frac{\partial^2F_{X,Y}(x,y)}{\partial x\partial y};\]
and for $t=t_1+{\rm i}t_2$, where $t_1,t_2\in\mathbb R$, the {\em CF} is given by\footnote{In studies of complex random variables, the CF is often defined as 
$\phi_Z(t) = \mathbb{E}\bigl[e^{{\rm i} \operatorname{Re}(\overline{t} Z)}\bigr]$ 
for $t \in \mathbb{C}$. For $t = t_1 + {\rm i} t_2$ and $Z = X + {\rm i} Y$, 
this equals $\mathbb{E}\bigl[e^{{\rm i}(t_1X + t_2Y)}\bigr]$, 
the joint CF of the real and imaginary parts.}
\begin{equation}
    \phi_Z(t) = \phi_{X,Y}(t_1, t_2) = \mathbb{E} \left[ e^{\mathrm{i}(t_1 X + t_2 Y)} \right], \quad (t_1, t_2) \in \mathbb{R}^2.
\end{equation}

For the {\em expectation} of $Z = X + {\rm i} Y$, we have
$$
\mathbb{E}[Z] = \mathbb{E}[X] + {\rm i} \mathbb{E}[Y].
$$
Whenever $\mathbb{E}[Z]$ exists, expectation and complex conjugation commute, that is, 
$$
\mathbb{E}[\overline{Z}] = \overline{\mathbb{E}[Z]}.
$$

For the {\em variance} of $Z = X + {\rm i} Y$, it is defined by
$$
\mathbb{V}(Z) = \mathbb{E}[|Z - \mu_Z|^2].
$$
Using $|z|^2 = z\overline{z}$, we have
\begin{align*}
\mathbb{V}(Z) &= \mathbb{E}\bigl[(Z-\mu_Z)(\overline{Z}-\overline{\mu_Z})\bigr] \\
&= \mathbb{E}[|Z|^2] - |\mu_Z|^2.
\end{align*}
Using $|z|^2 = x^2 + y^2$ with $Z-\mu_Z = (X-\mu_X) + {\rm i}(Y-\mu_Y)$, we obtain
\begin{align*}
\mathbb{V}(Z) &= \mathbb{E}[|X-\mu_X|^2 + |Y-\mu_Y|^2] \\
&= \mathbb{V}(X) + \mathbb{V}(Y).
\end{align*}

In the {\em polar approach}, write $Z(\omega) = R(\omega) e^{{\rm i} \Theta(\omega)}$ 
for $\omega \in \Omega$, where $R(\omega) > 0$, $0 \leq \Theta(\omega) < 2\pi$, 
and $R, \Theta$ are real-valued random variables on $(\Omega, \mathcal{F}, \mathbb{P})$. 
The distribution of $Z$ is then identified with the joint distribution of $(R, \Theta)$.

In fact, these two approaches are equivalent via the transformation
$$
X = R \cos \Theta, \quad Y = R \sin \Theta.
$$

\subsubsection{Complex Gaussian}\label{sec:complexgaussian}
Here we focus on {\em centered} complex-valued Gaussian random variables, 
i.e., those with zero mean. Write 
$
Z = X + {\rm i}Y,
$
where $X \sim \mathscr{N}_\mathbb{R}(0,\sigma_1^2)$ and $Y \sim \mathscr{N}_\mathbb{R}(0,\sigma_2^2)$ 
are independent. Then 
$
\mathbb{V}(Z) = \sigma_1^2 + \sigma_2^2,
$
so $Z \sim \mathscr{N}_\mathbb{C}(0,\sigma_1^2 + \sigma_2^2)$.

We pay special attention to the {\em equi-variance case} $\sigma_1 = \sigma_2 = \sigma$, 
where $Z \sim \mathscr{N}_\mathbb{C}(0,2\sigma^2)$. In particular, when $2\sigma^2 = 1$, 
$Z$ is called a {\em standard} complex Gaussian random variable; see \cite{GOO63}.

Let $Z = X + {\rm i}Y$, where $X$ and $Y$ are independent centered Gaussians 
with variance $\sigma^2$, i.e., $X,Y \sim \mathscr{N}_\mathbb{R}(0,\sigma^2)$. 
After normalization, $X/\sigma, Y/\sigma \sim \mathscr{N}_\mathbb{R}(0,1)$, so
$$
\sigma^{-2}|Z|^2 = \sigma^{-2}(X^2 + Y^2) = \left(\frac{X}{\sigma}\right)^2 + \left(\frac{Y}{\sigma}\right)^2 \sim \chi_2^2 = \mathscr{E}\left(\frac{1}{2}\right),
$$
both having PDF
$$
f_{\sigma^{-2}|Z|^2}(x) = \frac{1}{2} e^{-x/2} \boldsymbol{1}_{[0,\infty)}(x).
$$

Furthermore (see \autoref{sec:exponentialdistribution}), $|Z|^2$ follows the 
exponential distribution with rate parameter $\frac{1}{2\sigma^2}$, that is,
\begin{equation}\label{square}
|Z|^2 \sim \mathscr{E}\left(\frac{1}{2\sigma^2}\right).
\end{equation}

Similarly, for the distribution of $|Z|$, we have
\begin{equation*}
f_{|Z|}(x) = \frac{x}{\sigma^2} e^{-x^2/(2\sigma^2)} \boldsymbol{1}_{[0,\infty)}(x).
\end{equation*}
This is the PDF of the so-called {\em Rayleigh distribution} with parameter $\sigma$, 
denoted $|Z| \sim \mathscr{R}(\sigma)$; see \cite{Sid62}.

Let $Z = R e^{{\rm i} \Theta}$. From the preceding discussion, the radius $R$ follows 
the {Rayleigh distribution} with parameter $\sigma$, i.e., $R \sim \mathscr{R}(\sigma)$, 
and $R^2$ follows the {exponential distribution} with rate parameter $\frac{1}{2\sigma^2}$, 
i.e.,
\begin{equation*}
R^2 \sim \mathscr{E}\left(\frac{1}{2\sigma^2}\right).
\end{equation*}
Furthermore, the phase $\Theta$ is {\em uniformly distributed} on $[0,2\pi)$, 
i.e., $\Theta \sim \mathscr{U}[0,2\pi)$; see \cite{GGKS23}.

\subsection{Chernoff bound}
\begin{lemma}[Chernoff bound]\label{lemma:ccbound}
Let \(X\) be a nonnegative random variable. Then, for any \(a \ge 0\) and \(\lambda > 0\),
\[
\mathbb{P}(X \ge a)
\le e^{-\lambda a}\,\mathbb{E}\bigl[e^{\lambda X}\bigr].
\]
\end{lemma}

\begin{proof}
Fix any \(\lambda > 0\). Since \(x \mapsto e^{\lambda x}\) is strictly increasing, we have
\[
\{X \ge a\} = \{e^{\lambda X} \ge e^{\lambda a}\}.
\]
Applying Markov's inequality to the nonnegative random variable \(e^{\lambda X}\), we obtain
\begin{align*}
\mathbb{P}(X \ge a)
&= \mathbb{P}(e^{\lambda X} \ge e^{\lambda a}) \\
&\le \frac{\mathbb{E}[e^{\lambda X}]}{e^{\lambda a}} \\
&= e^{-\lambda a}\,\mathbb{E}[e^{\lambda X}].
\end{align*}
This completes the proof of \autoref{lemma:ccbound}.
\end{proof}

\newcommand{\etalchar}[1]{$^{#1}$}
\providecommand{\bysame}{\leavevmode\hbox to3em{\hrulefill}\thinspace}
\providecommand{\MR}{\relax\ifhmode\unskip\space\fi MR }
\providecommand{\MRhref}[2]{%
  \href{http://www.ams.org/mathscinet-getitem?mr=#1}{#2}
}
\providecommand{\href}[2]{#2}

\end{document}